\documentclass[11pt]{amsart}
\usepackage{amsmath,amsfonts,amssymb,amsthm}
\usepackage{latexsym,bm,graphicx}
\usepackage{mathrsfs}
\usepackage{color}
\usepackage{hyperref}
\usepackage{xypic}

\title[The moduli space of metrics with positive isotropic curvature]{Path-connectedness of the moduli spaces of metrics with positive isotropic curvature on four-manifolds}

\author{Bing-Long Chen}
\address{School of Mathematics and Computational science\\  Sun Yat-sen University\\ Guangzhou 510275\\ E-mail address: mcscbl@mail.sysu.edu.cn}

\author{Xian-Tao Huang}
\address{School of Mathematics and Computational science\\  Sun Yat-sen University\\ Guangzhou 510275\\ E-mail address: hxiant@mail2.sysu.edu.cn}

\newtheorem{thm}{Theorem}[section]
\newtheorem{prop}[thm]{Proposition}
\newtheorem{lem}[thm]{Lemma}

\newtheorem{cor}[thm]{Corollary}

\theoremstyle{definition}
\theoremstyle{remark}

\newtheorem{rem}[thm]{Remark}

\numberwithin{equation}{section}

\newcommand{\R}{\mathbb{R}}

\begin{document}
%\today

\begin{abstract}In this paper we prove the path connectedness of the moduli spaces of metrics with positive isotropic curvature on certain compact four-dimensional manifolds.

\vspace*{5pt}

\vspace*{5pt}

\end{abstract}

\maketitle

\section{Introduction}

Let $M$ be a compact $n$-dimensional smooth manifold. In Riemannian geometry, whether $M$ admits a metric with certain curvature restriction is a fundamental problem. These curvature conditions may include  positive scalar curvature, positive Ricci curvature, positive or negative sectional curvature, etc. Suppose $M$ admits such a metric, people are also interested in the topology of the spaces of all such Riemannian metrics on $M$.

Denote the set of Riemannian metrics $g$ with positive scalar curvature $R_{g}$ by $\mathcal{R}_{+}(M)$. The group of diffeomorphisms on $M$, denoted by $\textmd{Diff}(M)$, acts on $\mathcal{R}_{+}(M)$ naturally.  In 1916,
Weyl \cite{Weyl} proved that $\mathcal{R}_{+}(S^{2})$ is path-connected. Rosenberg and Stolz \cite{RosenSto} further showed  that  $\mathcal{R}_{+}(S^{2})$ is contractible.

When dimension $n\geq 7,$ there are many examples with disconnected $\mathcal{R}_{+}(M^{n})$ or even the moduli spaces $\mathcal{R}_{+}(M^{n})/\textmd{Diff}(M^{n})$, see \cite{Hit}, \cite{GL1}, \cite{Carr}, \cite{KS}, \cite{Rosen} etc.
%For example, Hitchin \cite{Hit} proved that the spaces $\mathcal{R}_{+}(S^{8k})$ and $\mathcal{R}_{+}(S^{8k+1})$ are disconnected for each $k\geq1$;
%Carr \cite{Carr} proved that the space $\mathcal{R}_{+}(S^{4k-1})$ has infinitely many connected components for each $k\geq2$, (in the $k=2$ case, this result was proved earlier by Gromov and Lawson \cite{GL1});
%Kreck and Stolz \cite{KS} showed that even the moduli space $\mathcal{R}_{+}(S^{4k-1})/\mathrm{Diff}(S^{4k-1})$ has infinitely many connected components if $k\geq2$.
However, when dimension $n=3$, Marques \cite{Mar} proved recently  that  the moduli space $\mathcal{R}_{+}(M)/\textmd{Diff}(M)$ is path-connected if $M$ is compact orientable and $\mathcal{R}_{+}(M)\neq\emptyset$. Combining the result of Cerf \cite{Cerf} on $\textmd{Diff}_{+}(S^3)$,  Marques \cite{Mar}  further argued that $\mathcal{R}_{+}(S^{3})$ is path-connected.

In dimension 4, we remark that the connectedness of  the moduli space $\mathcal{R}_{+}(S^{4})/\textmd{Diff}(S^{4})$ is still unknown. The structure of $\textmd{Diff}(S^{4})$ is also a big unsolved problem.

In this paper, instead, we consider the positive isotropic curvature condition on four-manifolds. The notion of positive isotropic curvature was introduced by Micallef and Moore \cite{MicMo}. This elegant curvature condition plays a key role in Brendle-Schoen's  proof of famous $1/4$-sphere theorem, see \cite{BreSch}. The recent research \cite{Ham2} \cite{CZh1} \cite{CTZh} shows that positive isotropic curvature in dimension 4 is analogues to  positive scalar curvature in dimension 3.
For instance, Schoen-Yau \cite{SchYau} and Perelman \cite{Per2} showed that a compact orientable  three-manifold admits a metric with positive scalar curvature if and only if it is diffeomorphic to a connected sum of orientable spherical space forms $S^3/\Gamma_i$ and quotients  $(S^2\times\mathbb{R})/G_j$, where the actions of $\Gamma_i$ and $G_j$ are standard.
In \cite{CTZh}, the authors proved that a compact four-manifold admits a metric with positive isotropic curvature if and only if it is diffeomorphic to $S^{4},\mathbb{RP}^{4},(S^{3}\times{\mathbb{R}})/G,$ or a connected sum of them, where the action of $G$ on $S^{3}\times{\mathbb{R}}$ is also standard isometric action.

Motivated by this result, we investigate  the space of metrics on $M$ with positive isotropic curvature. Denote this space by $\textmd{PIC}(M)$.
It is well known that $\textmd{PIC}(M)\subset\mathcal{R}_{+}(M)$.
The purpose of this paper is trying to prove that the moduli space $\textmd{PIC}(M)/\textmd{Diff}(M)$ is path-connected. The result is the following:

\begin{thm} \label{t1.1}
The moduli space $\textmd{PIC}(M)/\textmd{Diff}(M)$ is path-connected if $M$ is orientable and diffeomorphic to one of the following manifolds:
\begin{enumerate}
  \item $S^{4}$,
  \item $(S^{3}\times{\mathbb{R}})/G$, where $G$ is a cocompact fixed point free discrete subgroup of isometries of $S^{3}\times{\mathbb{R}}$,
  \item a finite connected sum $(S^{3}/\Gamma_{i}\times S^{1})\#\ldots\#(S^{3}/\Gamma_{k}\times S^{1})$, where $\Gamma_{i}$ $(1\leq{i}\leq{k})$ is either the trivial group or a non-cyclic isometric group of $S^{3}$.
\end{enumerate}
\end{thm}

%\begin{thm}
%The moduli space $\textmd{PIC}(M)/\textmd{Diff}(M)$ is path-connected if $M$ is orientable and diffeomorphic to one of the following manifolds:
%\begin{enumerate}
%  \item $S^{4}$,
%  \item a finite connected sum $(S^{3}\times S^{1})\#\ldots\# (S^{3}\times S^{1})$,
%  \item $(S^{3}\times{\mathbb{R}})/G$ or a finite connected sum of $S^{3}/\Gamma_{i}\times S^{1}$, $1\leq{i}\leq{k}$, where $G$ is a cocompact fixed point free discrete subgroup of isometries of $S^{3}\times{\mathbb{R}}$, and $\Gamma_{i}$ is either the trivial group or a non-cyclic isometric group of $S^{3}$.
%\end{enumerate}
%\end{thm}

The dissatisfaction of the above result is that not all manifolds with $\textmd{PIC}(M)\neq\emptyset$ have  been  handled. Let $g$, $g'$ be metrics with positive isotropic curvature. We call $g$ is isotopic to  $g'$ if there exists a continuous path $g_{\mu},\mu\in[0,1]$ such that $g_{0}=g$, $g_{1}=g'$ and $g_{\mu}\in\textmd{PIC}(M)$ for every $\mu\in[0,1]$.
We require that $\textmd{PIC}(M)$ is always equipped with the $C^{\infty}$-topology.  The path connectedness of  $\textmd{PIC}(M)/\textmd{Diff}(M)$ just means that for any two $g_1,g_2\in\textmd{PIC}(M)$, there is a diffeomorphism $\varphi$ such that $g_1$ is isotopic to $\varphi^{*}g_2.$

The proof of Theorem \ref{t1.1} is mainly using Ricci flow with surgery. The idea is to use Ricci flow to deform the initial metric $g_0=g$. Since Ricci flow preserves positive isotropic curvature, the solution $g_t$ is isotopic to $g_0$ until singularities were hit for the first time $t_1$.
Cutting off the higher curvature part, gluing back suitable caps, we find $g_{t_1}$ is \textquoteleft separated\textquoteright \text{} into several pieces. Running $g_{t_1^+}$ with the Ricci flow on the pieces with lower curvature again, for $t>t_1$, $g_t$ is isotopic to $g_{t_1^+}$ until it hits  some singularities again at some time $t_2>t_1$.
Repeating these procedures until all pieces are terminated. Remember that at each singular time $t_k$, the orbifold $M_{t_k^-}$  is decomposed into several pieces $M_{k_1},\cdots, M_{k_l}$. The orbifold $M_{t_k^-}$ can be obtained as an orbifold connected sum (See \cite{CTZh}, or Section \ref{s2}) of these $M_{k_i}$'s.
The point, here, is to show that the metric $g_{t_k^-}$ on $M_{t_k^-}$ is isotopic to a metric obtained through a canonical construction from those $(M_{k_i},g_{t_{k}^+})$'s. Trying to assemble all pieces together, one can finally show that the initial metric is isotopic to a canonical metric obtained from some standard pieces.

\begin{thm} \label{t1.2}
Let $M^{4}$ be  a compact four-dimensional manifold with $\textmd{PIC}(M)\neq\emptyset$.
If $g\in\textmd{PIC}(M)$, then there is a path of metrics $g_{\mu}, \mu\in[0,1]$ such that $g_{0}=g$, $g_{1}$ is a canonical metric, and $g_{\mu}\in\textmd{PIC}(M)$ for all $\mu\in[0,1]$.
\end{thm}
The precise definition of a canonical metric is given in Section \ref{s3}. The above  idea has been successfully used by Marques \cite{Mar} in dimension 3.  The main steps of our proof follows the line of arguments in \cite{Mar}.
To prove Theorem \ref{t1.1}, one has to prove the canonical metrics of Theorem \ref{t1.2} are unique module the group of diffeomorphisms. See Section \ref{s8} for the differences between dimension 3 and 4.

The organization of this paper is as follows.
In Section \ref{s2},  we fix some terminologies and discuss  Micallef and Wang's  connected sum construction of manifolds with PIC, and extend it to families of orbifold connected sums.
In Section \ref{s3}, we give the definition of canonical metrics used throughout this paper.
In Section \ref{s4},  we give  some preliminary results on deforming metrics with positive isotropic curvature.
In Section \ref{s5}, we recall  surgery process of Hamilton on orbifolds and discuss the isotopic property of the metric during the surgeries.
In Section \ref{s6},  we discuss the canonical neighborhood assumption.
In Section \ref{s7},  we prove Theorem \ref{t1.2}.
In Section \ref{s8},  we  prove Theorem \ref{t1.1}.

\vspace*{2em}

\noindent\textbf{Acknowledgments.} The authors  would like to thank Professors X.P.Zhu and  S. H. Tang for helpful discussions. The first author is partially supported by NSFC 11025107.

\section{Preliminaries and M-W connected sum}\label{s2}

\subsection{Positive isotropic curvature}\label{s2.1}

Let's recall the notion of positive isotropic curvature.

Let $(M^{n},g)$ be an $n$-dimensional Riemannian manifold, $n\geq4.$  $(M^n,g)$ is said to have positive isotropic curvature (PIC for short) if $R_{1313}+R_{1414}+R_{2323}+R_{2424}-2R_{1234}>0$ for any othornormal four vectors $e_1,e_2,e_3,e_4$. PIC was introduced by Micallef and Moore \cite{MicMo}, it appears naturally in the second variation formula of area of surfaces.  %There are many papers studying this curvature condition, see, for example, \cite{MicMo}, \cite{MicWang}, \cite{Ham2}, \cite{CTZh}, \cite{CZh1}, \cite{BreSch}, \cite{Ng}, \cite{Fra}, \cite{GaSe}.
PIC implies the positive scalar curvature condition (see \cite{MicWang}).
One of the most interesting thing is that PIC is preserved by Ricci flow (see \cite{Ham2}, \cite{BreSch}, \cite{Ng}).

Let $(M^{4},g)$ be a four-dimensional Riemannian manifold. The orientation gives  the bundle $\Lambda^{2}TM$ a decomposition  $\Lambda^{2}TM=\Lambda^{2}_{+}TM\bigoplus\Lambda^{2}_{-}TM$ into its self-dual and anti-self-dual parts. Therefore the curvature operator has a block decomposition
\[\mathcal{R}=\begin{pmatrix} A & B \\ B^{t} & C \end{pmatrix},\]
where $A = W_{+} + \frac{R}{12}I$, $C = W_{-} + \frac{R}{12}I,$ and $B$ is the traceless part of the Ricci curvature, $W_{\pm}$ are the self-dual and anti-self-dual Weyl tensors respectively.
Denote the eigenvalues of the matrices $A, C$ and $\sqrt{BB^{t}}$ by $a_{1}\leq a_{2}\leq a_{_{3}}, c_{1}\leq c_{2}\leq c_{3}, b_{1}\leq b_{2}\leq b_{3}$ respectively.
It is known that PIC is equivalent to $a_{1} + a_{2} > 0$ and $c_{1} + c_{2} > 0$ (see \cite{Ham2}).
From this, it is clear that if $g$ is locally conformally flat, then $g$ has positive scalar curvature if and only if $g$ has positive isotropic curvature.

If we define
\[\sigma_{g}:=R_{g}-6\max\{\lambda_{\max}(W_{+}), \lambda_{\max}(W_{-})\},\]
where $\lambda_{\max}(W_{\pm})$ are the largest eigenvalue of $W_{\pm}$ respectively. Since both $W_{+}$ and $W_{-}$ are trace free, it is easy to see that the condition $\sigma_{g}>0$ is equivalent to $a_{1} + a_{2} > 0$ and $c_{1} + c_{2} > 0$, that is to say, $(M^{4}, g)$ has positive isotropic curvature.

Let $\tilde{g}=u^{2}g, u\in C^{\infty}(M), u>0$, then by direct computation we can obtain the following relations:
\[R_{\tilde{g}}=u^{-3}(-6\Delta u+R_{g}u),\]
\[\sigma_{\tilde{g}}=u^{-3}(-6\Delta u+\sigma_{g}u).\]

\subsection{Orbifolds}\label{s2.2}

We will clarify some terminologies and notations about orbifolds.

For an $n$-dimensional orbifold $X^{n}$, $x \in X$, we use $\Gamma_{x}$ to denote the local uniformization group at $x$, that is, there is an open neighborhood $B_{x}\ni x$, such that $B_{x}=B^{n}/\Gamma_{x}$, where $B^{n}$ is diffeomorphic to $\mathbb{R}^{n}$ and $\Gamma_{x}$ is a finite subgroup of linear transformations fixing the origin.

By the Lefschetz fixed-point formula, every orientation-reversing diffeomorphism of $S^{3}$ has a fixed point (see \cite{Scott}).
Therefore, if $X$ is a four dimensional orbifold with at most isolated singularities, then for every point $x\in{X}$, $\Gamma_{x}\subset{SO(4)}$.

We will fix some notations of orbifolds that will appear in this paper.

Suppose $\Gamma$ is a fixed point free finite subgroup of $SO(4)$ acting on $S^{3}$. If we write the equation of $S^{4}$ as $x_{1}^{2}+\ldots+x_{5}^{2}=1$. Regard $S^3$ as an equator in $S^4,$  $\Gamma$ can naturally extend  to act isometrically on $S^{4}$ by fixing the $x_{1}$-axis. We will still use $\Gamma$ to denote this group action. The orbifold $S^{4}/\Gamma$ has exactly two fixed points $(1, 0, 0, 0, 0)$ and $(-1, 0, 0, 0, 0)$ with local uniformization group $\Gamma$.
In this paper, if we say a spherical orbifold is of the form $S^{4}/\Gamma$ with $\Gamma\subset{SO(4)}$, we always means this case.

If $S^{3}/\Gamma$ admits a fixed point free isometry $\tau$ satisfying $\tau^{2}=1$, then we can define an action $\hat{\tau}$ on $S^{3}/\Gamma\times{\mathbb{R}}$ by $\hat{\tau}(\theta,r)=(\tau(\theta),-r)$, where $\theta\in{S}^{3}/\Gamma,r\in{\mathbb{R}}$. The quotient $(S^{3}/\Gamma\times{\mathbb{R}})/\{1,\hat{\tau}\}$ is a smooth four manifold with a neck-like end $S^{3}/\Gamma\times{\mathbb{R}}$.
We denote this manifold by $C^{\tau}_{\Gamma}$.
If we think of $S^{4}$ as the compactification of $S^{3}\times{\mathbb{R}}$ by adding two points at infinities of $S^{3}\times{\mathbb{R}}$, then $\Gamma$ and $\hat{\tau}$ can be naturally  regarded as isometries of the standard $S^{4}$ .
We denote $S^{4}/\langle\Gamma,\hat{\tau}\rangle$ the resulting orbifold in this paper.
Obviously, $C^{\tau}_{\Gamma}$ is diffeomorphic to the smooth manifold obtained by removing the orbifold singularity (or a smooth point when $\Gamma$ is trivial) from $S^{4}/\langle\Gamma,\hat{\tau}\rangle$.
$\mathbb{RP}^{4}\setminus\bar{B}^{4}$ is an example of $C^{\tau}_{\Gamma}$.

In the following, we define topologically necks and caps whose meaning will be fixed in our subsequent discussions.  A neck is defined to be a manifold diffeomorphic to $S^{3}/\Gamma\times{\mathbb{R}}$.
For caps, we have smooth caps and orbifold caps. Smooth caps consist  of $C_{\Gamma}^{\tau}$ and $B^{4}$. Our orbifold caps have two types.

The orbifold cap of Type I is diffeomorphic to $\mathbb{R}^{4}/\Gamma$, where $\Gamma\subset{SO(4)}$ is a finite subgroup fixing the origin of $\mathbb{R}^{4}$ and acting freely on the unit three-sphere in $\mathbb{R}^{4}$.
We denote it by $C_{\Gamma}$.
$C_{\Gamma}$ has a neck-like end $S^{3}/\Gamma\times{\mathbb{R}}$ and one orbifold singularity with local uniformization group $\Gamma$.

The orbifold cap of type II is constructed as follows. Let the equation of $S^{3}$ be $x_{1}^{2}+\ldots+x_{4}^{2}=1$.
The isometry $\gamma:(x_{1},x_{2},x_{3},x_{4})\rightarrow(x_{1},-x_{2},-x_{3},-x_{4})$ has exactly two fixed points $p_{1}=(1, 0, 0, 0)$ and $p_{2}=(-1, 0, 0, 0)$, and satisfies $\gamma^2=1$.
We define an action $\hat{\gamma}$ on $S^{3}\times{\mathbb{R}}$ by $\hat{\gamma}(x,r)=(\gamma(x),-r)$, where $x\in{S}^{3},r\in{\mathbb{R}}$.It is clear  $\hat{\gamma}^2=1.$
Denote the quotient orbifold $(S^{3}\times{\mathbb{R}})/\{1,\hat{\gamma}\}$ by $C_{II}$, and call it the orbifold cap of Type II.
It has a neck-like end $S^{3}\times{\mathbb{R}}$ and two orbifold singularities $(p_{1},0),(p_{2},0)$ with local uniformization group $\mathbb{Z}_{2}$.
There is another way to understand $C_{II}$. Let  the equation of $S^{4}$ be $x_{1}^{2}+\ldots+x_{5}^{2}=1$, the isometry $\zeta:(x_{1},x_{2},x_{3},x_{4},x_{5})\rightarrow(x_{1},-x_{2},-x_{3},-x_{4},-x_{5})$ has exactly two fixed points $(1,0,0,0,0)$ and $(-1,0,0,0,0)$. Removing a smooth point from $S^{4}/\{1,\zeta\}$, we get an orbifold diffeomorphic to $C_{II}$.

\subsection{Orbifold connected sum} \label{s2.3}

Suppose $X_{1}, X_{2}$ are $n$-dimensional orbifolds (not necessarily distinct) with at most isolated singularities. Let $x_{1}\in X_{1}, x_{2}\in X_{2}$ be two distinct points (not necessarily singular) such that $\Gamma_{x_{1}}$ is conjugate to $\Gamma_{x_{2}}$ as linear subgroups.
Let $f$ be a diffeomorphism from $\partial{B}_{x_{1}}$ to $\partial{B}_{x_{2}}$, we remove $B_{x_{1}}$ and $B_{x_{2}}$ from the orbifolds, and identify the boundary $\partial{B}_{x_{1}}$ and $\partial{B}_{x_{2}}$ by the diffeomorphism $f$.
The resulting orbifold is denoted by $\#_{f;x_{1},x_{2}}(X_{1},X_{2})$ or $\#_{f}(X_{1},X_{2})$, and is called orbifold connected sum of $X_{1}$ and $X_{2}$. Note that the diffeomorphism type of the resulting orbifold depends only on the isotopy class of $f$.
When orientation is taken into account, we adopt the convention that the orientation of $\partial{B}_{x_{1}}$ is induced from the orientation of $X_{1}$, while the orientation of $\partial{B}_{x_{2}}$ is reverse to that induced from $X_{2}.$

It has to be mentioned that any diffeomorphim $f\in\textmd{Diff}(S^{3}/\Gamma)$ must be  isotopic to some isometry $f'\in\textmd{Isom}(S^{3}/\Gamma)$ by \cite{Mcc}.

Suppose $X$ is diffeomorphic to $S^{4}/\Gamma,$ $\Gamma\subset{SO(4)}$, with two orbifold singularities $p_{1}$ and $p_{2}$ (when $\Gamma$ is trivial, we take $p_{1}$ and $p_{2}$ to be arbitrary two different smooth points).
If we perform an orbifold connected sum on $X$ with itself at $p_{1}$ and $p_{2}$ by $f\in\textmd{Diff}(S^{3}/\Gamma)$, then we obtain the mapping torus of $f$, and denote it by $S^{3}/\Gamma\times_{f}{S}^{1}$.
$S^{3}/\Gamma\times_{f}{S}^{1}$ has the structure of a fiber bundle over $S^{1}$ with fibers $S^{3}/\Gamma$ and the monodromy $f$. It is easy to see that the bundle structure depends only on the isotopy class of  $f$.

%In Section \ref{SECconnsum} we will consider the connected sum of Micallef and Wang which %preserves the positive isotropic curvature condition.

\subsection{M-W connected sum}\label{SECconnsum}

Schoen and Yau \cite{SchYau},  Gromov and Lawson \cite{GL}  independently proved that the connected sum of two manifolds with positive scalar curvature still admits a metric with positive scalar curvature.
In the paper \cite{MicWang}, Micallef and Wang proved that the positive isotropic curvature condition is also preserved under connected sums. We give a short description of Micallef and Wang's construction here.

Let $(M^{n}_{1}, g_{1})$ be a Riemannian manifold with positive isotropic curvature. Given $p_{1}\in M_{1}$, denote $B_{r}(p_{1})$ the geodesic ball of radius $r$ around $p_{1}$. Let $r_{1}>\rho_{1}>0$ be two positive number. By careful computations (see \cite{MicWang} for details), Micallef and Wang showed  that we can find a nonincreasing smooth function $\alpha: (0,r_{1}]\rightarrow\mathbb{R}$ such that
\[2k_{1}+\frac{\alpha(\frac{3}{2}-\alpha)}{r^{2}}+\frac{\alpha'}{r} > 0,\]
\[\alpha(r)\equiv 1 \text{ near }0, \]
\[\alpha(r)\equiv 0 \text{ on }[r^{\ast}_{1},r_{1}],\]
where $k_{1}$ is the minimum of isotropic curvatures of $g_{1}$, $r^{\ast}_{1}<r_{1}$ is a sufficiently small constant depending on $g_{1}$.
If we take $u(r)=\exp(\int_{r}^{r_{1}}(\alpha(x)/x) dx)$ in $B_{r_{1}}(p_{1})$, and $u\equiv1$ outside $B_{r_{1}}(p_{1})$, then the metric $g'_{1}=u^{2}g_{1}$ is a complete metric with positive isotropic curvature on $M\setminus\{p_{1}\}$, is the same as $g_{1}$ outside $B_{r_{1}}(p_{1})$, and near $p_{1}$ becomes $C^{2}$ close to the product metric $ds^{2}+\rho^{2}d\theta^{2}$ on $S^{n-1}\times\mathbb{R}^{+}$ for some sufficiently small $\rho<\rho_{1}$, where $d\theta^{2}$ is the metric on $S^{n-1}$ with constant curvature $1$.
We can slightly modify $g'_{1}$ at the half-cylinder end with a cut-off function to obtain a metric $\tilde{g}_{1}$ with positive isotropic curvature such that it is the product metric $ds^{2}+\rho^{2}d\theta^{2}$ on $S^{n-1}\times\mathbb{R}^{+}$ when $s$ is large enough.
%If $g_{1}$ is locally conformally flat, we can make $\tilde{g}_{1}$ locally conformally flat by suitable choice of a cutoff function in the above construction.

Suppose we have another manifold $(M^{n}_{2},g_{2})$ with positive isotropic curvature and $p_{2}\in M_{2}$, we can do the same construction as above such that the new metric $\tilde{g}_{2}$ also has the same half-cylinder end with metric $ds^{2}+\rho^{2}d\theta^{2}.$
We  cut the two ends when $s$ is large enough and glue $M^{n}_{1}\setminus\{p_{1}\}$ and $M^{n}_{2}\setminus\{p_{2}\}$ together by an $f\in\textmd{Isom}(S^{n-1}).$ We obtain a manifold  $\#_{f;p_{1},p_{2}}(M_{1},M_{2})$ with  a metric $\#_{f;p_{1},p_{2}}(g_{1},g_{2})$ of positive isotropic curvature.

%If $g_{1}$ and $g_{2}$ are locally conformally flat, we can make $\sharp_{f;p_{1},p_{2}}(g_{1},g_{2})$ locally conformally flat, too. This is because, by making use of the locally conformally flat structure, we can choose the function $\alpha$ such that

%To achieve this we should make suitable choice of the parameters $r_{1},\rho_{1}$ and the function $\alpha$ carefully.  and the cut-off argument to half-cylinder ends must be more carefully to matain the locally conformally flat structure.

\begin{rem}

Clearly, the resulting metric $\#_{f;p_{1},p_{2}}(g_{1},g_{2})$ in the above construction is not unique.
It depends on our choice of the small parameters $r_{1},\rho_{1}$ and the function $\alpha$ and the isomorphism $f.$ But the resulting metrics with all continuous different choices are isotopic.

\end{rem}

Since the inequalities the function $\alpha$ satisfied depend only on the lower bound of the isotropy curvature, the above construction can also applied to a continuous family of metrics with PIC.
The construction above is local, we can apply it to orbifold connected sums. Therefore, the above construction can be generalized to the case of a families of orbifolds. The precise description is given in the following remark.

\begin{rem}\label{remark2.2}
Let $(X_{1},g_{1,\mu}),(X_{2},g_{2,\mu})$, $\mu\in[0,1]$ be two continuous paths of metrics with PIC on compact orbifolds $X_{1}$ and $X_{2}$ with at most isolated singularities.
Suppose $x_{1,\mu}\in X_{1}, x_{2,\mu}\in X_{2}$ are two family of points (not necessarily singular) such that $\Gamma_{x_{1,\mu}}$ is conjugate to $\Gamma_{x_{2,\mu}}$ as linear subgroups and independent of $\mu$ (so we can denote them by $\Gamma$).
Then one can deform these two families of metrics to get $(X_{i}\setminus\{x_{i,\mu}\},\tilde{g}_{i,\mu})$ with PIC, $i=1,2$, such that outside small geodesic balls $(B_{i}(x_{i,\mu},r),g_{i,\mu})$ they are the same as $(X_{i},g_{i,\mu})$, and near $x_{i,\mu}$ the metrics are the product metrics $ds^{2}+\rho^{2}d\theta^{2}$ on $S^{n-1}/\Gamma\times\mathbb{R}^{+}$ for some $\rho\ll r$.
For a continuous path of isometrics $f_{\mu}\in\textmd{Isom}(S^{n-1}/\Gamma)$, one can form a continuous path of metrics $\#_{f_{\mu};x_{1,\mu},x_{2,\mu}}(g_{1,\mu},g_{2,\mu}),{\mu\in[0,1]}$ with PIC on a fixed  orbifold $X\cong \#_{f_{\mu};x_{1,\mu},x_{2,\mu}}(X_{1},X_{2})$.
In addition, we can apply the same construction to a fixed $(X,g_{\mu})$ with $x_{1,\mu},x_{2,\mu}\in X$, and obtain a family of metrics $\#_{f_{\mu};x_{1,\mu},x_{2,\mu}}(g_{\mu})$ with PIC on a fixed $Y\cong \#_{f_{\mu};x_{1,\mu},x_{2,\mu}}(X)$ for $\mu\in[0,1]$.

\end{rem}

We will call such a procedure an \textmd{M-W} connected sum for short. % The necks appearing in the \textmd{M-W} procedure to connect each component will be called connected sum necks.
We summarize the above result into a proposition:
\begin{prop}\label{orbconnsum}
The \textmd{M-W} connected sum can be performed continuously on orbifolds, such that the resulting metrics have PIC and vary continuously with the parameters.
\end{prop}
\begin{rem}\label{r2.4}
Let $(M_1,g_1)$ and $(M_2,g_2)$ be two locally conformally flat four-manifolds (or orbifolds) with positive scalar curvature. There is another well-known connected sum construction $(M_1\#M_2,g)$ from $g_1$ and $g_2$ so that the resulting metric $g$ is conformally flat with positive scalar curvature (see \cite{SchYau} \cite{Hoe}). Note that $(M_i,g_i)$ also has PIC for $i=1,2,$ so we may also construct an \textmd{M-W} connected sum  $g'$ from $g_1$ and $g_2.$  The latter construction is performed at two small  geodesic balls, while the former is performed at two Euclidean balls conformal to the original metrics, we require  the gluing map $f\in Iso(S^3/\Gamma)$ in these two constructions lying in the same mapping class. Here, we remark that  two metrics $g$ and $g'$ obtained from the above two different constructions are isotopic.
\end{rem}

\section{Standard metrics and canonical metrics}\label{s3}

\subsection{Standard metrics}Let $h_{std}$ be the standard cylindrical metric of constant scalar curvature $1$ on $S^{3}\times{\mathbb{R}}$.
Let $G$ be a cocompact fixed point free discrete subgroup of the isometry group of $(S^{3}\times{\mathbb{R}},h_{std})$.
We denote $\Gamma=G\bigcap\textmd{Isom}(S^{3}\times\{0\})$.
Recall that every orientation-reversing diffeomorphism of $S^{3}$ has a fixed point.
Since the action $G$ is discrete and free, $\Gamma$ is finite, and the action of $\Gamma$ on $S^{3}$ is orientation-preserving.
Clearly, $G/\Gamma$ acts isometrically on $S^{3}/\Gamma\times{\mathbb{R}}$ with the quotient $(S^{3}\times{\mathbb{R}})/G$, and $(G/\Gamma)\bigcap\textmd{Isom}(S^{3}/\Gamma\times\{0\})=\emptyset$.
We denote $p_{2}$ the projection from $\textmd{Isom}(S^{3}/\Gamma\times{\mathbb{R}})=\textmd{Isom}(S^{3}/\Gamma)\times\textmd{Isom}(\mathbb{R})$ into the second factor.
Since $G/\Gamma$ acts cocompactly on $S^{3}/\Gamma\times{\mathbb{R}}$, $p_{2}(G)$ is a discrete subgroup of $\textmd{Isom}(\mathbb{R})$, hence is $\mathbb{Z}$ or $D(\infty)$, where $D(\infty)$ denote the isometry group of $\mathbb{R}$ generated by  two reflections $\varsigma_{1}:r\rightarrow{-r}$ and $\varsigma_{2}:r\rightarrow{r_1-r},$ with $r_1\neq 0.$

If $p_{2}(G/\Gamma)=\mathbb{Z}$, let $\varsigma_{0}$ be a generator of $\mathbb{Z}$, then $\rho=p_{2}^{-1}(\varsigma_{0})$ satisfying $\rho:(\theta,r)\rightarrow(f(\theta),r+r_0)$, with $r_0>0$, $f\in\textmd{Isom}(S^{3}/\Gamma)$.
In this case, $(S^{3}\times\mathbb{R})/G$ is obtained from identifying the boundaries of $S^{3}/\Gamma\times[0,r_0]$ by the isometry $f\in\textmd{Isom}(S^{3}/\Gamma)$, i.e. $(S^{3}\times{R})/G$ is of the form $S^{3}/\Gamma\times_{f}S^{1}$.

If $p_{2}(G/\Gamma)=D(\infty)$, let  $\rho_{1}=p_{2}^{-1}(\varsigma_{1})$ and $\rho_{2}=p_{2}^{-1}(\varsigma_{2})$ be two inverses of the reflecting generators of $D(\infty).$ As isometries of  $S^{3}/\Gamma\times\mathbb{R}$, $\rho_1$ and $\rho_2$ have the following forms $\rho_{1}:(\theta,r)\rightarrow(\tau_{1}(\theta),-r)$, $\rho_{2}:(\theta,r)\rightarrow(\tau_{2}(\theta),r_1-r)$.
Since the action $G/\Gamma$ is free, $\tau_{1},\tau_{2}\in\textmd{Isom}(S^{3}/\Gamma)$ are orientation-preserving involutions.
Therefore, in this case the quotient $(S^{3}\times\mathbb{R})/G$ is not orientable. Note that by Section \ref{s2.2}, $S^{4}/\langle\Gamma,\hat{\tau}_{i}\rangle$ are spherical orbifolds with one singularity of local uniformization group  $\Gamma.$  It is clear that $(S^{3}\times\mathbb{R})/G$ is diffeomorphic to an orbifold connected sum of $S^{4}/\langle\Gamma,\hat{\tau}_{1}\rangle$ and $S^{4}/\langle\Gamma,\hat{\tau}_{2}\rangle.$

Suppose $M$ is a differentiable manifold obtained by identifying the boundaries of $S^{3}/\Gamma\times[0,r_0]$ by $f\in\textmd{Diff}(S^{3}/\Gamma)$.
Since $f\in\textmd{Diff}(S^{3}/\Gamma)$ is isotopic to an isometry $f'\in\textmd{Isom}(S^{3}/\Gamma)$ by \cite{Mcc}, $M$ is diffeomorphic to $S^{3}/\Gamma\times_{f'}S^{1}$, and admits a geometry structure induced by an isometric group action $G$ on $S^{3}\times\mathbb{R}$ as described above.
Similarly, if $M$ is obtained by performing an orbifold connected sum of $S^{4}/\langle\Gamma,\hat{\tau}_{1}\rangle$ and $S^{4}/\langle\Gamma,\hat{\tau}_{2}\rangle$ at the singularities, then $M$ also admits a geometry structure induced by an isometry group action $G$ on $S^{3}\times\mathbb{R}$.
The induced metric from $S^{3}\times\mathbb{R}$ on $M$ is also called a standard metric, and is denoted by $h_{std}$.

\begin{rem} \label{remark-diff-iso-1}

We note that standard metrics on such $M$ are not unique.  It depends on the diffeomorphism from $M$ to some $(S^{3}\times\mathbb{R})/G$.
\end{rem}

Let $h_{round}$ be the standard metric of constant curvature $1$ on the sphere $S^{4}$. If $\Gamma$ is a discrete subgroup of isometries of  $S^{4}$,  we will also call the induced metric on $S^{4}/\Gamma$ a standard metric, and still denote it by $h_{round}$.

Suppose $S^{4}/\tilde{\Gamma}$ is a spherical orbifold with at most isolated singularities.
By studying the corresponding group actions on $S^{4}$ (see Lemma 5.1 and 5.2 in \cite{CTZh}), one can derive that such a $S^{4}/\tilde{\Gamma}$ has no more than two orbifold singularities.
So there are three possible types of $S^{4}/\tilde{\Gamma}$ according to the number of singularities:
the first type is $S^{4}$ or $\mathbb{RP}^{4}$, the second type is of the form $S^{4}/\langle\Gamma,\hat{\tau}\rangle$ with $\Gamma\subset{SO(4)}$, the third case is of the form $S^{4}/\Gamma$ with $\Gamma\subset{SO(4)}$.

In summary, we only defined the standard metrics on manifolds or orbifolds of the following types:\\
\text{}\ \ \ i) $S^{3}/\Gamma\times_{f}S^{1},$ $S^{4}/\langle\Gamma,\hat{\tau}_{1}\rangle\#_f S^{4}/\langle\Gamma,\hat{\tau}_{2}\rangle$,\\
\text{}\ \ \ ii) $S^{4}/\tilde{\Gamma}.$

\subsection{Canonical metrics}

In the following, we will define canonical metrics on certain four-manifolds or orbifolds.

Let $p_{1},\ldots,p_{k},q_{1},\ldots,q_{l}$ be $(k+l)$ distinct points on $(S^{4},h_{round})$.
Let $(M_{i},h_{i}),i\in\{1,\ldots,k\}$ be compact manifolds diffeomorphic to  $(S^{3}\times\mathbb{R})/G_{i}$ with standard metrics, with $p_{i}'\in M_{i}$.
Let $(X_{j},\tilde{h}_{j}),j\in\{1,\ldots,l\},$ be spherical orbifolds $S^{4}/\Gamma_{j}$ with isolated singularities and with standard metrics, where $\Gamma_{j}\subset O(5)$, $\Gamma_{j}\neq\{1\}.$ Let $q_{j}'\in X_{j}$ be smooth points of $X_j,$ for $j=1,2,\cdots,l$.
If we perform the \textmd{M-W} connected sum operation (see Section \ref{SECconnsum} for definition)  on $(S^{4},h_{round},p_{1},\cdots,p_{k},q_{1},\cdots,q_{l})$ and those $(M_{i},h_{i},p_i')$ and $(X_{j},\tilde{h}_{j},q_j'),$   we obtain a Riemannian orbifold $(X,\hat{g})$.
We will call such $\hat{g}$ a canonical metric. Clearly, the resulting metric $\hat{g}$ has positive isotropic curvature.
%\begin{rem}\label{rem3.2}
%From the construction, since each component is actually locally conformally flat with positive scalar curvature, by Remark \ref{r2.4}, $\hat{g}$ is isotopic to a locally conformally flat with positive scalar curvature .
%\end{rem}

The sphere $(S^{4},h_{round})$ in the above construction will be called a principal sphere.
$(M_{i},h_{i})$ and $(X_{j},\tilde{h}_{j})$ in the above construction will be called subcomponents of $X$.
The decomposition of $X$ as the connected sum of subcomponents of above types is called a canonical decomposition of $X$ in this paper.

Since the resulting metric $\hat{g}$ depends on the choice of parameters in the \textmd{M-W} connected sum process, canonical metrics constructed above are not unique.
However, the different choice of parameters in the \textmd{M-W} connected sum process will produce an isotopy between canonical metrics.

The orbifolds admitting a canonical decomposition will naturally appear in the process of Ricci flow with surgery. We will prove that our initial manifold $(M,g)$ admits a canonical decomposition, and $g$ is isotopic to a canonical metric.

In paper \cite{Mar}, Marques also gave the notion of a canonical metric, where he allowed to perform connected sum on the principle sphere $S^{3}$ with itself, and the other subcomponents are spherical three-manifolds.
In dimension four, we have to face with orbifold singularities, and we do not allow to perform connected sums on the principle sphere $S^{4}$ with itself in a canonical metric, and we consider the subcomponents of the form $(S^{3}\times{R})/G$ so that we have a uniform description.

%\begin{rem}

%From the construction, a canonical metric has positive isotropic curvature and is locally conformally flat.

%\end{rem}

\begin{rem}

Here we note that a canonical metric on $(S^{3}\times\mathbb{R})/G$ (or $S^{4}/\Gamma$ with $\Gamma\neq\{1\}$) is not the same as the standard metric induced from the $(S^{3}\times\mathbb{R},h_{std})$ (or $(S^{4},h_{round})$). The reason is that we always require the presence of the principle sphere in the definition of canonical metrics. However, we will prove that a canonical metric on such orbifolds is always  isotopic to a standard metric in Section \ref{s4}.

\end{rem}

\section{Deforming canonical metrics} \label{s4}

\subsection{Deforming standard metrics}

In this section,  we will give some preliminary results on deforming metrics with positive isotropic curvature.

\begin{prop}\label{p4.2}

Let $(M, g)$ be a compact four orbifold with positive isotropic curvature. Then
the space $\textmd{PIC}(M)\bigcap\{\tilde{g}|\tilde{g}=u^{2}g, u\in C^{\infty}(M), u>0\}$ of metrics is star-shaped, hence contractible.

\end{prop}

\begin{proof}

Recall that we have defined $\sigma_{g}:=R_{g}-6\max\{\lambda_{\max}(W_{+}), \lambda_{\max}(W_{-})\}$ in Section \ref{s2}, and we know $\sigma_{g}>0$ is equivalent to PIC.

Let $\tilde{g}=u^{2}g, u\in C^{\infty}(M), u>0$, we have:
\[\sigma_{\tilde{g}}=u^{-3}(-6\Delta u+\sigma_{g}u).\]
Denote $u_{\mu}=(1-\mu+\mu{u}), g_{\mu}=u_{\mu}^{2}g, \mu\in[0,1]$, $\sigma_{\mu}=\sigma_{g_{\mu}}$, then $g_{0}=g, g_{1}=\tilde{g}$.
Suppose both $\sigma_{g}$ and $\sigma_{\tilde{g}}$ are positive, we have
\[u_{\mu}^{3}\sigma_{\mu}=\mu(-6\Delta u+\sigma_{g}u)+(1-\mu)\sigma_{g}>0.\]
Hence every $g_{\mu}$ has positive isotropic curvature and we finish the proof.
\end{proof}

\begin{prop}\label{p4.3}

Let $(M, g)$ be a compact locally conformally flat four orbifold with positive scalar curvature.
Suppose $(M,\tilde{g})$ is obtained from an \textmd{M-W} connected sum of $(M, g)$ with a sphere $(S^4,h_{round})$. Then $\tilde{g}$ is isotopic to $g$.

\end{prop}
\begin{proof} Let $g'$ be a locally conformally flat metric with positive scalar curvature obtained from the connected sum of $(M,g)$ and $(S^4,h_{round})$ in the sense of \cite{SchYau} and \cite{Hoe} (see Remark \ref{r2.4}). Then $g'$ is conformal to $g,$ hence isotopic to $g$ by Proposition \ref{p4.2}. On the other hand, by Remark \ref{r2.4},  $\tilde{g}$ is isotopic to $g'$, hence to $g$.
\end{proof}
\begin{prop}\label{p4.4}

Let $(M,g)$ be  $S^{4}/G,$ $S^{3}/\Gamma\times_{f}S^{1},$ or $S^{4}/\langle\Gamma,\hat{\tau}_{1}\rangle\#_f S^{4}/\langle\Gamma,\hat{\tau}_{1}\rangle, $ equipped with a standard metric, $(M,\tilde{g})$ be the canonical metric obtained from an \textmd{M-W} connected sum of $(M^{4}, g)$ with a sphere $S^4.$  Then $\tilde{g}$ is isotopic to $g$.
\end{prop}
\begin{proof} The standard metrics are locally conformally flat with positive scalar curvature. The result follows from Proposition \ref{p4.3}.
\end{proof}
We remark that the Proposition \ref{p4.4} does not claim the following statement: any canonical metric on an orbifold diffeomorphic to one of the above orbifolds in Proposition \ref{p4.4} is isotopic to a standard one. Nevertheless, at least, we can handle the spherical case at present:

\begin{prop}\label{p4.6}

\textbf{Any} canonical metric on an orbifold diffeomorphic to $S^{4}/G$ is isotopic to standard metric.
\end{prop}

From the construction of canonical metrics and the van Kampen theorem for orbifolds (see \cite{BMP}), we know in the canonical decomposition of $S^{4}/G$, there is at most one nontrivial standard piece, which is diffeomorphic to $S^{4}/G$ itself. Hence Proposition \ref{p4.6} follows from Proposition \ref{p4.4}.

The following proposition provides another proof of Proposition \ref{p4.6}, it also has its own interest.

\begin{prop}\label{p4.5}

Let $(S^{4}/G,g)$ be an orbifold, where $G$ is a finite subgroup of $O(5)$. Suppose $g$ is locally conformally flat with PIC, then $g$ is isotopic to a constant curvature metric. \end{prop}

\begin{proof}

We pull back  $g$ to a $G$-invariant, locally conformally flat metric $\tilde{g}$ with PIC on $S^{4}$. By the works of Kuiper (\cite{Kui1}), $\tilde{g}$ is conformally equivalent to a constant curvature metric.
Recall that in dimension four, a locally conformally flat metric has positive scalar curvature if and only if it has PIC.
We solve the Yamabe flow on $S^{4}$:
\[\frac{\partial{g(t)}}{\partial{t}}=-(R_{g(t)}-r_{g(t)})g(t),\text{     }g(0)=\tilde{g},\]
where $r_{g(t)}$ is the mean value of the scalar curvature.
Since the Yamabe flow preserves the conformal structure, $g(t)$ is locally conformally flat for every $t$. By the evolution equation of $R_{g}$, one can derive that the positive scalar curvature condition is preserved by the Yamabe flow.
Hence, in this case, the Yamabe flow preserves the positive isotropic curvature condition.
By the result of Ye \cite{Ye} (see also \cite{Bre}), if the initial metric on $S^{n}$ is conformally equivalent to the standard metric, then the Yamabe flow has a global solution which converges exponentially to a metric of constant sectional curvature.
Because the initial metric $\tilde{g}$ is $G$-invariant, by the uniqueness of the solution of the Yamabe flow, Yamabe flow descents to  a flow on the orbifold $S^{4}/G.$  This produces a desired deformation.
\end{proof}

Let $d\theta^{2}$ be a constant curvature metric of scalar curvature $1$ on a spherical manifold $S^{3}/\Gamma$, where $\Gamma$ is a finite subgroup of $SO(4)$ acting freely on $S^{3}$.

\begin{prop}\label{p4.7}

Let $g=dr^{2}+\omega^{2}(r)d\theta^{2}$ be a warped product metric on ${S}^{3}/\Gamma\times(-a,a)$ with PIC and $\omega(r)\equiv1$ on $(-a,-b)\bigcup(b,a),$ $0<b<a.$
Then there exists a continuous path of warped product  metrics $g_{\mu}, \mu\in[0,1]$ with PIC, such that $g_{0}=g$, $g_{1}=dr^{2}+d\theta^{2}$, $g_{\mu}\equiv g$ on ${S}^{3}/\Gamma\times(-a,-b)\cup (b,a).$
In addition, if $g$ is symmetric under reflection, i.e.  $\omega(r)=\omega(-r)$ for $r\in(-a,a)$, then $g_{\mu}$ can also be chosen to be symmetric under reflection.

\end{prop}

\begin{proof}

Let $\tilde{g}=\omega(r)^{-2}g$, $\upsilon(r)=\int_{0}^{r}\omega^{-1}(x)dx$, then $\tilde{g}=d\upsilon^{2}+d\theta^{2}$, which clearly has positive isotropic curvature.

For $\mu\in[0,\frac{1}{2}]$, let $g_{\mu}=(1-2\mu+2\mu{\omega}^{-1}(r))^{2}g$.
Then $g_{0}=g, g_{\frac{1}{2}}=\tilde{g}$, $g_{\mu}=g$ on $S^{3}/\Gamma\times(-a,-b)$ and $S^{3}/\Gamma\times(b,a)$ for every $\mu\in[0,\frac{1}{2}].$ It is obvious that  $g_{\mu}$ is warped product(not necessarily in the same $(r,\theta)$ coordinates).  From the proof of Proposition \ref{p4.2}, $g_{\mu}$ has positive isotropic curvature for every $\mu\in[0,\frac{1}{2}]$.

For $\mu\in[\frac{1}{2},1]$, let $g_{\mu}=[(2-2\mu)\omega^{-2}(r)+2\mu-1]dr^{2}+d\theta^{2}$. Then $g_{\frac{1}{2}}=\tilde{g}, g_{1}=dr^{2}+d\theta^{2}$, $g_{\mu}=g$ on $S^{3}/\Gamma\times(-a,-b)$ and $S^{3}/\Gamma\times(b,a)$ for every $\mu\in[\frac{1}{2},1]$.
It's obvious that $g_{\mu}$  is warped product and has PIC on $S^{3}/\Gamma\times(-a,a)$ for every $\mu\in[0,1]$.

If $\omega(r)=\omega(-r)$ for $r\in(-a,a)$, then from the above construction $g_{\mu}$ is symmetric with respect to the reflection on the $r$-coordinate for every $\mu\in[0,1]$.
\end{proof}

Let $g=dr^{2}+\omega^{2}(r)d\theta^{2}$ be a warped product metric on $S^{3}/\Gamma\times\mathbb{R}$ with PIC. We assume $\omega(r)$ is invariant under a cocompact discrete isometric subgroup $G$ of $S^{3}/\Gamma\times\mathbb{R}$. That is to say, either $\omega(r)$ is periodic or $\omega(r)$ is periodic and reflectively invariant. By the same argument as in the proof of Proposition \ref{p4.7},
there exists a continuous path of metrics $g_{\mu}, \mu\in[0,1]$, such that $g_{0}=g$, $g_{1}=dr^{2}+d\theta^{2}$, $g_{\mu}$ is warped product and has PIC for all $\mu\in[0,1]$,
and $g_{\mu}$ is also invariant under $G.$  By taking quotient of  $(S^{3}/\Gamma\times\mathbb{R},g_{\mu})$ by $G,$ we obtain:

\begin{prop}\label{p4.8}

If $S^{3}/\Gamma\times_{f}S^{1}$ admits a PIC metric $g$ induced from a warped product  $dr^{2}+\omega^{2}(r)d\theta^{2}$ on $S^{3}/\Gamma\times\mathbb{R}$,
where $\omega$ is a period function satisfying $\omega(r)=\omega(r+2\pi)$,
then there exists a continuous path of PIC metrics $g_{\mu}$, $\mu\in[0,1]$, on $S^{3}/\Gamma\times_{f}S^{1}$ such that $g_{0}=g$, $g_{1}$ is a standard metric.

\end{prop}

\begin{prop}\label{p4.9}

Suppose $\#_{f}(S^{4}/\langle\Gamma,\hat{\tau}_{1}\rangle,S^{4}/\langle\Gamma,\hat{\tau}_{2}\rangle)$ admits a PIC metric $g$ induced from a warped product metric $dr^{2}+\omega^{2}(r)d\theta^{2}$ on $S^{3}/\Gamma\times\mathbb{R}$, where $\omega$ is a function satisfying $\omega(r)=\omega(-r)$ and $\omega(r)=\omega(r_0-r)$, $r_0\neq0$.
Then there exists a continuous path of PIC  metrics $g_{\mu}$, $\mu\in[0,1]$, on $\#_{f}(S^{4}/\langle\Gamma,\hat{\tau}_{1}\rangle,S^{4}/\langle\Gamma,\hat{\tau}_{2}\rangle)$ such that $g_{0}=g$, $g_{1}$ is a standard metric.

\end{prop}

\subsection{Deforming  canonical metrics}

%As application of the above propositions, we obtain the following:

%\begin{prop}

%Suppose $g$ is a standard metric on $S^{4}/\Gamma$ with $\Gamma\subset{SO(4)}$.
%We perform \textmd{M-W} connected sum at the two orbifold singularities $p_{1}$ and $p_{2}$ (in the case $\Gamma=\{1\}$, $p_{1}$ and $p_{2}$ are both smooth points) with $f\in\textmd{Isom}(S^{3}/\Gamma).$ Then the resulting metric $\tilde{g}=\#_{f;p_{1},p_{2}}(g)$ is isotopic to a standard  metric  on $S^{3}/\Gamma\times_{f}S^{1}$.

%\end{prop}

%\begin{proof}
%The result follows easily from Proposition \ref{p4.8}.
%\end{proof}

%\begin{prop}

%Suppose $(X_1,g_1)$ is $S^{4}/\Gamma$ or $S^{4}/\langle\Gamma,\hat{\tau}_1\rangle$, and $(X_2,g_2)$ is also  $S^{4}/\Gamma$ or $S^{4}/\langle\Gamma,\hat{\tau}_2\rangle$, both equipped with standard metrics.
%We perform \textmd{M-W} connected sum on orbifold singularities $p_{1}\in X_1$ and $p_2\in X_2$ with $f\in\textmd{Isom}(S^{3}/\Gamma)$. Then the resulting metric $\tilde{g}=\#_f(g_1,g_2)$ is isotopic to a standard  metric on $S^{4}/\Gamma$, $S^{4}/\langle\Gamma,\tilde{\tau}\rangle$ or $\#_f(S^{4}/\langle\Gamma,\hat{\tau}_1\rangle, S^{4}/\langle\Gamma,\hat{\tau}_2\rangle)$.

%\end{prop}
%\begin{proof}
%The proof is a simple application of Propositions \ref{p4.5} and \ref{p4.9}.
%\end{proof}

\begin{prop}\label{p4.10}

Suppose $g$ is a canonical metric on $S^{4}/\Gamma$ with $\Gamma\subset{SO(4)}$.
We perform \textmd{M-W} connected sum at the two orbifold singularities $p_{1}$ and $p_{2}$ (in the case $\Gamma=\{1\}$, $p_{1}$ and $p_{2}$ are both smooth points) with $f\in\mathrm{Isom}(S^{3}/\Gamma).$ Then the resulting metric $\tilde{g}=\#_{f;p_{1},p_{2}}(g)$ is isotopic to a standard  metric  on $S^{3}/\Gamma\times_{f}S^{1}.$

\end{prop}

\begin{proof} Note that $g$ is isotopic to some $h_{round}$ by Proposition \ref{p4.6}. Let $\bar{g}$ be the metric obtained from \textmd{M-W} connected sum on $(S^{4}/\Gamma,h_{round})$  at the two orbifold singularities $p_{1}$ and $p_{2}$.
By Proposition \ref{orbconnsum}, $\bar{g}$ is isotopic to $\tilde{g}$.
Since $\bar{g}$ is PIC and warped product, by Proposition \ref{p4.8}, it is isotopic to a standard metric on $S^{3}/\Gamma\times_{f}S^{1}$.
The result follows from Proposition \ref{orbconnsum}.
\end{proof}

\begin{prop}\label{p4.11}

Suppose $(X_1,g_1)$ is $S^{4}/\Gamma$ or $S^{4}/\langle\Gamma,\hat{\tau}_1\rangle$, and $(X_2,g_2)$ is also  $S^{4}/\Gamma$ or $S^{4}/\langle\Gamma,\hat{\tau}_2\rangle$, both equipped with canonical metrics.
We perform \textmd{M-W} connected sum on orbifold singularities $p_{1}\in X_1$ and $p_2\in X_2$ with $f\in\textmd{Isom}(S^{3}/\Gamma)$. Then the resulting metric $\tilde{g}=\#_f(g_1,g_2)$ is isotopic to a standard  metric on $S^{4}/\Gamma$, $S^{4}/\langle\Gamma,\tilde{\tau}\rangle$ or $\#_f(S^{4}/\langle\Gamma,\hat{\tau}_1\rangle, S^{4}/\langle\Gamma,\hat{\tau}_2\rangle)$.

\end{prop}
\begin{proof}
The proof is similar to Proposition \ref{p4.10}.
\end{proof}

\begin{prop}\label{p4.12}

Let $g$ be a canonical metric on an orbifold $X$. If we perform \textmd{M-W} connected sum at two smooth points of $X$ by an $f\in\textmd{Isom}(S^{3})$, then the resulting metric is isotopic to an \textmd{M-W} connected sum of $(X,g)$ and some $(S^{3}\times_{f}S^{1},h_{std})$.

\end{prop}

\begin{proof}

By Proposition \ref{orbconnsum}, we can assume  those two smooth points where connected sum is performed are on the principal sphere of $X$. The result follows from Proposition \ref{p4.10} and Proposition \ref{orbconnsum}.
\end{proof}

The following theorem is the main result of this section:

\begin{prop} \label{p4.13}

Suppose $(X_{1},g_{1}),\ldots,(X_{l},g_{l})$ are orbifolds endowed with canonical metrics. If $(X,g)$ is the resulting Riemannian orbifold by performing \textmd{M-W} connected sums between them (we allow performing connected sums on some $X_{i}$ with itself), then $g$ is isotopic to a canonical metric on $X$.

\end{prop}

\begin{proof}

By induction on $l,$ we only need to consider $l=2$ case. Suppose \textmd{M-W} connected sum is performed at two points $p_1\in X_1,$ $p_2\in X_2$ or $X_1.$ We divide our argument into 4 cases:\\
\text{}\ \ \ i) $p_1\in X_1,$ $p_2\in X_2$ and $p_1,p_2$ are smooth,\\
\text{}\ \ \ ii) $p_1\in X_1,$ $p_2\in X_2$ and $p_1,p_2$ are singular,\\
\text{}\ \ \ iii) $p_1,p_2\in X_1$ and $p_1,p_2$ are smooth,\\
\text{}\ \ \ iv) $p_1,p_2\in X_1$ and $p_1,p_2$ are singular.

Case i): By Proposition \ref{orbconnsum}, we may assume $p_1$ and $p_2$ lie in the principle spheres.  Note that by the same argument of Proposition \ref{p4.11},  \textmd{M-W} connected sum of two spheres is isotopic to one sphere.  From this fact and Proposition \ref{orbconnsum},  we can merge the two principal spheres gradually  to  one sphere and this simultaneously produces an isotopy of $g$ on $X$ to some canonical metric.

Case ii): Let $p_1$($p_2$) lie in some subcomponent $SX_1$ ($SX_2$). Let $(X_1',g_1')$ ($(X_2',g_2')$) be the orbifold with a canonical metric  such that $X_1$ is an \textmd{M-W} connected sum of $SX_1$ ($SX_2$) with $X_{1}'$($X_2'$) at its principle sphere. In this case $SX_{1}$ and $SX_{2}$ are spherical orbifolds of the form $S^{4}/\Gamma$ or $S^{4}/\langle\Gamma,\hat{\tau}\rangle$ with $\Gamma\subset{SO(4)}$.
By Proposition \ref{p4.12} an \textmd{M-W} connected sum of  $SX_{1}$ and $SX_{2}$ at their singular points is isotopic to a standard metric on  $X_3=\#_f(SX_{1},SX_{2})$, where the possible topological type of $X_3$ is $S^{4}/\Gamma$, $S^4/\langle \Gamma,\hat{\tau}\rangle$ or $\#_{f}(S^4/\langle \Gamma,\hat{\tau_1}\rangle,S^4/\langle \Gamma,\hat{\tau_2}\rangle)$.
This isotopy shows that the metric $g$ is isotopic to an \textmd{M-W} connected sum between a \textbf{canonical metric} on $X_1'\# X_3$ and a canonical metric on $X_2'$ at \textbf{smooth} points.  This reduces to the  case i).

Case iii): This case has been treated by Proposition \ref{p4.12}.

Case iv): If $p_1$ and $p_2$ lie in one subcomponent $SX_1$, then $SX_1$ can only has the form $S^{4}/\Gamma$, with $\Gamma\subset{SO(4)}.$  By Propositions \ref{orbconnsum} and \ref{p4.8}, $g$ is isotopic to a canonical metric obtained as an \textmd{M-W} connected sum of a canonical metric on $X_1'$ and $S^3/\Gamma\times_f S^2$ at the principle sphere.
If $p_1$ and $p_2$ lie in different subcomponents, by similar argument as in case ii), these two subcomponents can merge together to give a standard piece. Our metric $g$ is isotopic to an \textmd{M-W} connected sum of a canonical metric on $X_{3}$ at two smooth points. This reduces to case iii).
\end{proof}

\section{Surgery} \label{s5}

Ricci flow with surgery was introduced first by R. Hamilton in \cite{Ham2}, where the author gave a detailed description on  how to do surgery on Ricci flow.  We need  one crucial property of this construction, namely, preserving positivity of the curvature operator.

Consider the standard round cylinder $S^{3}\times\mathbb{R}$ with the metric $h_{std}$ of scalar curvature $1$. Denote the coordinate of the second factor by $s$.
Let $f$ be a smooth function defined by
\begin{equation}\label{5.1}
f(s)=
    \left\{
    \begin{array}{ll}
        0,&s\leq0,\\
        ce^{-\frac{q}{s}},&s>0,
    \end{array}\right.
\end{equation}
where $c>0,q>0.$

Suppose $h$ is a metric on $N=S^{3}\times(-4,4),$ which is $\epsilon$-close to $h_{std}$ in $C^{[\frac{1}{\epsilon}]}$-topology. Let  $\hat{h}=e^{-2f}h$ be a new metric. Throughout the paper, without further indications,  $\epsilon$-close is always understood in $C^{[\frac{1}{\epsilon}]}$-topology.

\begin{prop}\label{p5.1}

There exist positive constants $c_{0}$, $q_{0}$ and $\epsilon_{2}$ with the following property. For any  $c<c_{0}$ and $q>q_{0}$ in (\ref{5.1}), and $0<\epsilon<\epsilon_{2},$ if
  $h$ is a metric on $N=S^{3}\times(-4,4)$ which is $\epsilon$-close to $h_{std}$ and has \textbf{positive curvature operator},
then the  metric $\hat{h}=e^{-2f}h$ also has \textbf{positive curvature operator}.

\end{prop}

Here we remark that the property of $\hat{h}$ stated in Proposition \ref{p5.1} was  not used in the usual theory on four-dimensional Ricci flow \cite{Ham2} \cite{CZh1} \cite{CTZh}.  The result in dimension three appeared in \cite{MT}. Since there is no further literature for four dimensional case, we give  the details for completeness.

\begin{proof} [Proof of Proposition \ref{p5.1}]

First we fix the constant $q_{0}\gg16$ such that $\frac{q^{2}}{s^{4}}e^{-\frac{q}{s}}\ll 1$ for every $q>q_{0}$ and $s\in(0,4]$. We will fix $c_{0}$ to be a small constant, such that for every $c<c_{0}$ and $q>q_{0}$,  \[\frac{q}{s^{2}}f^{2}\ll\frac{q^{2}}{s^{4}}f^{2}\ll\frac{q}{s^{2}}f\ll1,\forall s\in(0,4].\]

Let $\{E_{0},E_{1},E_{2},E_{3}\}$ be an orthonormal basis for the tangent space at a point $(x,s)\in{N}=S^{3}\times(-4,4)$ with respect to the metric $h_{std}$ such that $E_{0}$ points to the positive $s$-direction.
We denote the covariant derivative with respect to ${h}$ by ${\nabla}$,
and denote the matrix of the curvature operator of ${h}$ with respect to this basis by $\mathcal{{R}}=({R}_{ijkl})$.
Since $h$ is $\epsilon$-close to $h_{std}$, at a point $(x,s)\in N$, $h_{ij}=\delta_{ij}+O(\epsilon)$, and the Christoffel symbols $\Gamma^{i}_{jk}$ of $h$ are within $\epsilon$ in the $C^{[\frac{1}{\epsilon}]-1}$-topology of those of $h_{std}$. Therefore, with respect to the basis $\{E_{0},E_{1},E_{2},E_{3}\}$, $f_{0}=\frac{q}{s^{2}}f$ and $f_{i}=0$ for $1\leq i\leq3$, and
\[|\nabla f|_{h}=\frac{q}{s^{2}}f(1+O(\epsilon)),\]
\[f_{00}=(\frac{q^{2}}{s^{4}}-\frac{2q}{s^{3}})f+\frac{q}{s^{2}}fO(\epsilon),\]
\[f_{i0}=\frac{q}{s^{2}}fO(\epsilon),\text{ for }1\leq i\leq3,\]
\[f_{ij}=\frac{q}{s^{2}}fO(\epsilon),\text{ for }1\leq i,j\leq3,\]
where $f_{i}$ means $\nabla_{i}f$, $f_{ij}=(\mathrm{Hess}f)_{ij}=\partial_{i}f_{j}-f_{k}\Gamma^{k}_{ij}$ is the Hessian of $f$ with respect to the metric $h$.

Let's consider $\hat{h}=e^{-2f}h$. %We denote by $\hat{\nabla}$ the covariant derivative for $\hat{h}$.
We denote by $\mathcal{\hat{R}}=(\hat{R}_{ijkl})$ the matrix of the curvature operator of $\hat{h}$ with respect to the basis $\{E_{0},E_{1},E_{2},E_{3}\}$.

The following formula for the curvature of a conformally changed metric is well known:
\[
\begin{split}
\hat{R}_{ijkl}=e^{-2f}[R_{ijkl}-f_{j}f_{k}h_{il}+f_{j}f_{l}h_{ik}+f_{i}f_{k}h_{jl}-f_{i}f_{l}h_{jk}\\
-(\wedge^{2}h)_{ijkl}|\nabla{f}|^{2}-f_{jk}h_{il}+f_{ik}h_{jl}+f_{jl}h_{ik}-f_{il}h_{jk}],
\end{split}
\]
where $(\wedge^{2}h)_{ijkl}=h_{ik}h_{jl}-h_{il}h_{jk}$.
Hence, if $1\leq i,j,k,l\leq3$, we have
\[
\hat{R}_{0i0j}=e^{-2f}[R_{0i0j}+(\frac{q^{2}}{s^{4}}-\frac{2q}{s^{3}})fh_{ij}+\frac{q}{s^{2}}fO(\epsilon)],
\]
\[
\hat{R}_{ijk0}=e^{-2f}[R_{ijk0}+\frac{q}{s^{2}}fO(\epsilon)],
\]
\[
\hat{R}_{ijij}=e^{-2f}[R_{ijij}-\frac{q^{2}}{s^{4}}f^{2}+\frac{q}{s^{2}}fO(\epsilon)],
\]
\[
\hat{R}_{ijkl}=e^{-2f}[R_{ijkl}+\frac{q}{s^{2}}fO(\epsilon)] \text{ (for }\{i,j\}\neq\{k,l\}).
\]

Therefore, with respect to the basis $\mathcal{B}_0=$
$\{E_{0}\wedge{E_{1}},E_{0}\wedge{E_{2}},E_{0}\wedge{E_{3}},E_{1}\wedge{E_{2}},E_{1}\wedge{E_{3}},E_{2}\wedge{E_{3}}\}$, the (0,4) curvature tensor $\mathcal{\hat{R}},$ regarded as a quadratic form on $\bigwedge^2,$ has the following form:
\[
e^{-2f}\left[\mathcal{R}+\begin{pmatrix} (\frac{q^{2}}{s^{4}}-\frac{2q}{s^{3}})f\mathbf{I} & \mathbf{O} \\ \mathbf{O} & -\frac{q^{2}}{s^{4}}f^{2}\mathbf{I} \end{pmatrix}+\left(\frac{q}{s^{2}}fO(\epsilon)\right)\right],
\]
where $\mathbf{O}$ and $\mathbf{I}$ are the zero and  $3\times3$ identity matrix respectively.

Since $h$ is $\epsilon$-close to $h_{std}$, there exists a basis $\{F_{0},F_{1},F_{2},F_{3}\}$ which is orthonormal with respect to the metric $h$ such that:\\
\text{}\ \ \ i) the matrix of basis change from $\{E_{0},E_{1},E_{2},E_{3}\}$ to $\{F_{0},F_{1},F_{2},F_{3}\}$ is  $\mathbf{I}+(O(\epsilon)),$\\
\text{}\ \ \ ii) the curvature $\mathcal{R}$ with respect to the basis $\mathfrak{B}_{1}=\{F_{0}\wedge{F_{1}},F_{0}\wedge{F_{2}},F_{0}\wedge{F_{3}},F_{1}\wedge{F_{2}},F_{1}\wedge{F_{3}},F_{2}\wedge{F_{3}}\}$ of $\bigwedge^{2}TN$ is
\[
\mathcal{R}=
\begin{pmatrix} \mathbf{O} & \mathbf{O} \\ \mathbf{O} & \frac{1}{6}\mathbf{I} \end{pmatrix}+(O(\epsilon)).
\]
 Moreover, we may find a new orthogonal basis $\mathfrak{B}_{2}$ of $\bigwedge^{2}TN$ such that the matrix of basis change from $\mathfrak{B}_{2}$ to $\mathfrak{B}_{1}$ is a $6\times6$ matrix $\mathcal{A}_{1}=\mathbf{I}+O(\epsilon)$ and
\[
\mathcal{A}_{1}^{t}\mathcal{R}\mathcal{A}_{1}=
\begin{pmatrix} \mathbf{B_{1}} & \mathbf{O} \\ \mathbf{O} & \mathbf{B_{2}} \end{pmatrix},
\]
with $\mathbf{B_{1}}=(O(\epsilon))$, $\mathbf{B_{2}}=\frac{1}{6}\mathbf{I}+(O(\epsilon))$.  Denote the smallest eigenvalue of $\mathbf{B_{1}}$ by $\lambda$, which is also the smallest eigenvalue of the matrix $\mathcal{A}_{1}^{t}\mathcal{R}\mathcal{A}_{1}$. Obviously, $\lambda=O(\epsilon)$.
On the other hand, with respect to the basis $\mathfrak{B}_{2}$, $\mathcal{\hat{R}}$ is of the form
\[
e^{-2f}\left[
\begin{pmatrix} \mathbf{B_{1}} & \mathbf{O} \\ \mathbf{O} & \mathbf{B_{2}} \end{pmatrix}
+\begin{pmatrix} (\frac{q^{2}}{s^{4}}-\frac{2q}{s^{3}})f\mathbf{I} & \mathbf{O} \\ \mathbf{O} & -\frac{q^{2}}{s^{4}}f^{2}\mathbf{I} \end{pmatrix}
+\left(\frac{q^{2}}{s^{4}}fO(\epsilon)\right)\right]=e^{-2f}\mathbf{A}.
\]
 For a vector $\mathbf{x}=\begin{pmatrix} \mathbf{x_{1}} \\ \mathbf{x_{2}} \end{pmatrix}$ with $\mathbf{x_{1}},\mathbf{x_{2}}\in\mathbb{R}^{3}$, if $\epsilon$ is small enough, we have
\[
\begin{split}
\mathbf{x^{T}}\mathbf{A}\mathbf{x}&=\mathbf{x_{1}^{T}}\mathbf{B_{1}}\mathbf{x_{1}}+\mathbf{x_{2}^{T}}\mathbf{B_{2}}\mathbf{x_{2}}+[(\frac{q^{2}}{s^{4}}
-\frac{2q}{s^{3}})f]\mathbf{x_{1}^{T}x_{1}}-(\frac{q^{2}}{s^{4}}f^{2})\mathbf{x_{2}^{T}x_{2}}+(\frac{q^{2}}{s^{4}}fO(\epsilon))\mathbf{x^{T}x}\\
&\geq[\lambda+(\frac{q^{2}}{s^{4}}-\frac{2q}{s^{3}})f-\frac{q^{2}}{s^{4}}fO(\epsilon)]\mathbf{x_{1}^{T}x_{1}}
+[\frac{1}{6}-O(\epsilon)-\frac{q^{2}}{s^{4}}f^{2}-\frac{q^{2}}{s^{4}}fO(\epsilon)]\mathbf{x_{2}^{T}x_{2}}\\
&>(\lambda+\frac{q^{2}}{2s^{4}}f)\mathbf{x^{T}x}.
\end{split}
\]

By assumption, $h$ has positive curvature operator, hence $\lambda>0$. Therefore, the matrix $e^{-2f}\mathbf{A}$ is positive definite,
%Notice that $\{e^{f}F_{0},e^{f}F_{1},e^{f}F_{2},e^{f}F_{3}\}$ is orthonormal with respect to $\hat{h}$.
%It is clear that the matrix of $\mathcal{R}_{\hat{h}}$ with respect to this basis equals to $e^{4f}\mathbf{A}.$
i.e. $\hat{h}$ has positive curvature operator. The proof is completed.
\end{proof}

On the other hand, Hamilton showed  that if $c$ is small enough and $q$ large enough (independent of $\epsilon$), the metric $\hat{h}=e^{-2f}h$ satisfies Hamilton's improved pinching estimates \cite{Ham2}  on $s\in[0,4]$, and has positive curvature operator for $s\in(1,4]$ if $h$ is $\epsilon$-close to $h_{std}$ with $\epsilon$ sufficiently small.
From now on we will fix a small $c$ and a large $q$ satisfying Proposition \ref{p5.1} and Hamilton's result.
 It is easy to extend the metric $(S^{3}\times(-\infty,4],e^{-2f}h_{std})$ to a rotationally symmetric metric on $\mathbb{R}^{4}$ by gluing a suitable chosen nice cap with positive curvature operator and satisfying Hamilton's improved pinching estimates. We will fix such a metric, denote it by $g_{std}$. The fixed point $p$ of the $SO(4)$-action on $(\mathbb{R}^{4},g_{std})$ is called the tip.
From the construction, there is a positive constant $A_{0}$ such that outside the geodesic ball $B(p,A_{0})$ in $(\mathbb{R}^{4},g_{std})$, $(\mathbb{R}^{4}\setminus{B}(p,A_{0}),g_{std})$ is isometric to a half cylinder. We denote the isometry by $\psi_{std}:(\mathbb{R}^{4}\setminus{B}(p,A_{0}),g_{std})\rightarrow(S^{3}\times(-\infty,0],h_{std})$.
We also define a function $s_{std}:(\mathbb{R}^{4},g_{std})\rightarrow\mathbb{R}$ to be $s_{std}(x)=A_{0}-\mathrm{dist}_{g_{std}}(x,p)$.

Let $\alpha:\mathbb{R}\rightarrow[0,1]$ be a cutoff function such that $\alpha(s)\equiv1$ if $s\leq2$ and $\alpha(s)\equiv0$ if $s\geq3$.
There exists a universal constant $\epsilon_{3}$, if $h$ is $\epsilon$-close to $h_{std}$ with $\epsilon<\epsilon_{3}$, then the metric $\check{h}=e^{-2f}[\alpha(s)h+(1-\alpha(s))h_{std}]$ has positive curvature operator on $s\in(1,4]$, and $\check{h}$ has positive curvature operator on $s\in(0,1]$ provided ${h}$ has positive curvature operator.
Since $\check{h}$ equals to $e^{-2f}h_{std}$ on $s\in[3,4]$, we can glue the corresponding portion of $(\mathbb{R}^{4},g_{std})$ to it and obtain a new manifold $\mathcal{S}$, diffeomorphic to $\mathbb{R}^{4}$, with a new metric $h_{surg}$. This is the so-called Hamilton's surgery process.

From the construction above, we have:

\begin{prop} \label{p5.2}

Suppose $h$ is a metric on $S^{3}\times(-4,4)$ such that $h$ is $\epsilon$-close to $h_{std}$ in the $C^{[\frac{1}{\epsilon}]}$-topology with $\epsilon<\epsilon_{3}$. If $h$ has positive curvature operator, so does $h_{surg}$.

\end{prop}

\begin{rem} \label{r5.3}

From the construction above, it's clear that if $0<\epsilon<\epsilon_{3}$, and $h_{\mu},\mu\in[0,1]$, is a continuous path of metrics $\epsilon$-close to $h_{std}$ on $S^{3}\times(-4,4)$, then $(h_{\mu})_{surg}$ is a continuous path of metrics with positive isotropic curvature on $\mathcal{S}$.

\end{rem}

The following lemma concerns the deformation of metrics on surgery caps.

\begin{lem} \label{l5.4}

There exits a small constant $\epsilon_{4}$ with the following property. If $0<\epsilon<\epsilon_{4}$ and $h$ is a metric on $S^{3}\times(-4,4)$ such that $h$ is $\epsilon$-close to $h_{std}$,
then there exists a continuous path of metrics $h_{\mu}$, $\mu\in[0,1]$, with positive isotropic curvature on $\mathcal{S}$ such that $h_{0}=h_{surg}$, $h_{1}$ is rotationally symmetric, and $h_{\mu}$ restricts to the linear homotopy $(1-\mu)h+\mu{h}_{std}$ on $s^{-1}((-4,0))$ for all $\mu\in[0,1]$.

\end{lem}

\begin{proof}

Define $\bar{h}_{\mu}=(1-\mu)h+\mu{h}_{std}$, and $h_{\mu}=(\bar{h}_{\mu})_{surg}$ for $\mu\in[0,1]$. It's clear that $(\bar{h}_{1})_{surg}$ is rotationally symmetric.
\end{proof}

We can also perform the same surgery process as above on a neck $((S^{3}/\Gamma)\times(-4,4),h)$ to obtain a  Riemannian orbifold $(\mathcal{S},h_{surg})$, where $\mathcal{S}$ is diffeomorphic to $C_{\Gamma}$. Proposition \ref{p5.2} and Lemma \ref{l5.4} are also valid in this case.

Let $h$ be a Riemannian metric on $S^{3}/\Gamma\times(-4,4)$ which is $\epsilon$-close to the standard metric $h_{std}=ds^{2}+d\theta^{2}$, where $d\theta^{2}$ is a constant curvature metric of scalar curvature $1$ on $S^{3}/\Gamma$.
Let $(\mathcal{S}^{\pm},h_{surg}^{\pm})$ be the orbifolds obtained by cutting along the central slice $S^{3}/\Gamma\times\{0\}$ and gluing surgery caps to both sides of the central slice, respectively.
We can perform the \textmd{M-W} connected sum at the tips $p^{\pm}$ of the surgery caps and obtain $(\mathcal{S}^{+}\#\mathcal{S}^{-},h_{surg}^{+}\#{h}_{surg}^{-})$ with PIC.

\begin{lem} \label{l5.5}

There is a universal constant $\epsilon_{5}$ with the following property. For any $0<\epsilon<\epsilon_{5}$, if $(S^{3}/\Gamma\times(-4,4),h)$ is $\epsilon$-close to $h_{std}$, there is a diffeomorphism from $(S^{3}/\Gamma\times(-4,4),h)$ to $\mathcal{S}^{+}\#\mathcal{S}^{-}$ such that the pulled back of the metric $h_{surg}^{+}\#{h}_{surg}^{-}$ can be deformed to $(S^{3}/\Gamma\times(-4,4),h)$ through PIC metrics  which all coincide with $h$ on the regions $s^{-1}((-4,-3))$ and $s^{-1}((3,4))$.

\end{lem}

\begin{proof}

Let $\alpha:\mathbb{R}\rightarrow[0,1]$ be a cutoff function such that $\alpha(s)\equiv1$ if $|s|\leq2$ and $\alpha(s)\equiv0$ if $|s|\geq3$. Let $h_{\mu}=(1-\mu\alpha(s))h+\mu\alpha(s)h_{std}$, $\mu\in[0,1]$.
Note that $h_{\mu}=h$ for all $\mu\in[0,1]$ if $|s|\geq3$, and $h_{1}=h_{std}$ if $s\in[-2,2]$.
If $h$ is $\epsilon$-closed to $h_{std}$ with $\epsilon$ small enough, then
by Remark \ref{r5.3}, and Proposition \ref{orbconnsum}, $(\mathcal{S}^{+}\#\mathcal{S}^{-},h_{surg}^{+}\#{h}_{surg}^{-})$ can be deformed to $(\mathcal{S}^{+}\#\mathcal{S}^{-},(h_{1})_{surg}^{+}\#(h_{1})_{surg}^{-})$ through PIC metrics which coincide with $h$ near the two ends of $\mathcal{S}^{+}\#\mathcal{S}^{-}$.
On the other hand, we can reparameterize $(\mathcal{S}^{+}\#\mathcal{S}^{-},(h_{1})_{surg}^{+}\#(h_{1})_{surg}^{-})$ to get  $(S^{3}/\Gamma\times(-4-a,4+a),\tilde{h})$ for some suitable $a>0$ such that, when $r\in(-4-a,-a)\bigcup(a,4+a)$, $\tilde{h}$ coincide with $h_{1}$ on $s^{-1}((-4,0)\bigcup(0,4))$; when $r\in(-2-a,2+a)$, $\tilde{h}=dr^{2}+\omega^{2}(r)d\theta^{2}$; when $r\in(-2-a,-a)\bigcup(a,2+a)$, $\omega\equiv1$.
By Proposition \ref{p4.7},  we can deform the metric $\tilde{h}$ on the region $r\in(-2-a,2+a)$ to $h_{std}$ through PIC  metrics with fixed  two ends.
Thus we can deform $(\mathcal{S}^{+}\#\mathcal{S}^{-},(h_{1})_{surg}^{+}\#(h_{1})_{surg}^{-})$ to $(S^{3}/\Gamma\times(-4,4),h_{1})$ through PIC metrics with fixed  two ends.
Finally we can deform $h_{1}$ back to $h$ by the deformation $h_{\mu}$.
\end{proof}

\section{Canonical neighborhood assumption} \label{s6}

\subsection{Ricci flow with surgery}

The Ricci flow equation
\[\frac{\partial g}{\partial t}=-2Ric_{g}\]
was introduced by Hamilton \cite{Ham1}. Now let $g_0$ be a metric with PIC on a compact four-manifold. We evolve $g_0$ by the Ricci flow and get a solution $g_t$ for $t>0$. Since the scalar curvature is positive, the solution will blow up in finite time. After performing surgeries, the solution may be resumed; repeating these procedures many times till we get nothing. This will produce a solution $g_t$ to the Ricci flow with surgery (see \cite{CZh1} \cite{CTZh}). Gluing all pieces left by surgeries together, at length, the following theorem can be proven (see \cite{CTZh}):

\begin{thm}\label{t6.1} {\rm(Main Theorem in \cite{CTZh})}\indent
Let $(M^{4}, g)$ be a compact four-dimensional manifold. Then it admits a metric with positive isotropic curvature if and only if it is diffeomorphic to $S^{4},\mathbb{RP}^{4},(S^{3}\times\mathbb{R})/G$ or a connected sum of them. Here $G$ is a cocompact fixed point free discrete isometric subgroup  of the standard $S^{3}\times\mathbb{R}$.

\end{thm}

In this paper, we are concerned not only the topology  but also the geometry of the manifold $M$.  So we have to investigate closely the geometric structure of the pieces left by surgeries.
In order to understand the singularities of the Ricci flow solution,
we need to investigate the ancient $\kappa$-solution $(X, g_{t})$, $t\in(-\infty,0]$, on an four-orbifold $X$ with at most isolated singularities, such that $g_{t}$ is PIC and satisfies the so called the restricted isotropic curvature pinching condition
\[a_{3} \leq\Lambda a_{1}, c_{3} \leq\Lambda c_{1}, b_{3}^{2}\leq a_{1}c_{1},\]
where $\Lambda$ is a positive constant.
As in \cite{CTZh}, we call such a solution an ancient $\kappa$-orbifold solution for short.

\subsection{Ancient solution }

The classification of ancient $\kappa$-orbifold solutions was done in Chapter 3 of \cite{CTZh}, we state the results in the following theorem:

\begin{thm}\label{t6.2} {\rm(Theorem 3.2-3.8 in \cite{CTZh})}\indent
For a four-dimensional ancient $\kappa$-orbifold solution $(X, g_{t})$, $t\in(-\infty,0]$,

\begin{enumerate}
  \item if the curvature operator has nontrivial null eigenvector somewhere, then $X$ is isometric to $S^{3}/\Gamma\times\mathbb{R}$, $C_{\Gamma}^{\tau}$ or $C_{II}$, with the induced metric from the product metric on $S^{3}\times\mathbb{R}$,
  \item if the curvature operator is strictly positive everywhere, then either $X$ is compact and diffeomorphic to a spherical  space form $S^{4}/\Gamma$ with at most isolated singularities, or $X$ is noncompact and diffeomorphic to $\mathbb{R}^{4}$ or $C_{\Gamma}$.
\end{enumerate}
Furthermore, for $\epsilon>0$ small enough one can find positive constants $C_{1}=C_{1}(\epsilon)$, $C_{2}=C_{2}(\epsilon)$, such that for every $(x, t)$, there is a radius $r$, $\frac{1}{C_{1}}(R(x,t))^{-\frac{1}{2}}<r<C_{1}(R(x,t))^{-\frac{1}{2}}$, so that some open neighborhood $B_{t}(x,r)\subset B\subset B_{t}(x, 2r)$ falls into one of the following categories:

\begin{description}
  \item[(i)] $B$ is an evolving $\epsilon$-neck around $(x, t)$  (in the sense that it is the time slice at time $t$ of the parabolic region $\{(x',t')|x'\in B, t'\in[t-R(x, t)^{-1}, t]\}$, where the solution is well defined on the whole parabolic neighborhood and is, after scaling with factor $R(x, t)$ and shifting the time $t$ to zero, $\epsilon$-close to the corresponding subset of the evolving round cylinder $S^{3}/\Gamma\times\mathbb{R}$ with scalar curvature $1$ at the time zero),
  \item[(ii)] $B$ is an evolving $\epsilon$-cap (in the sense that it is the time slice at the time $t$ of an evolving metric on open caps $\mathbb{R}^{4}$, $C^{\tau}_{\Gamma}$, $C_{\Gamma}$ or $C_{II}$ such that the region outside some suitable compact subset is an evolving $\epsilon$-neck),
  \item[(iii)] at time $t$, $x$ is contained in a connected compact component with positive curvature operator.
\end{description}
Moreover, the scalar curvature of $B$ in cases $(i)$ and $(ii)$ at time $t$ is between $C_{2}^{-1}R(x,t)$ and $C_{2}R(x,t)$.
\end{thm}

Usually, an $\epsilon$-cap in a four-orbifold $(X^{4},g)$ is defined to be an open submanifold $\mathcal{C}\subset X$ such that $\mathcal{C}$ is topologically a cap diffeomorphic to $\mathbb{R}^{4}$, $C_{\Gamma}^{\tau}$, $C_{\Gamma}$ or $C_{II}$, and there are open sets $N,Y\subset\mathcal{C}$ satisfying the following:
$N$ is an $\epsilon$-neck, and $\overline{Y}=\mathcal{C}\setminus N$ is a compact submanifold with boundary, and $\partial\overline{Y}$ is a central slice of some $\epsilon$-neck.
The set $Y$ in the above definition is called the core of $\mathcal{C}$.
The above information is usually sufficient for topological applications, but in this paper we are dealing with the properties of metrics, so we require more information about the geometry of an $\epsilon$-cap.

In order to do this, we discuss some properties of a noncompact ancient $\kappa$-orbifold solution with positive curvature operator.

\begin{prop}\label{p6.3}

There exists a universal function $\lambda:(0,\infty)\rightarrow(0,\infty)$ with the following property.
Let $(X,g_{t})$ $(t\in(-\infty,0])$ be a noncompact ancient $\kappa$-orbifold solution with positive curvature operator, then at every time $t_{0}$, there exits a point $x_0$ such that for any $y\in X$, the smallest eigenvalue of the curvature operator ${\mathcal{R}}$ at $(y,t_0)$ is greater than $\lambda(d_{t_0}(x_{0},y)R(x_{0},t_0)^{\frac{1}{2}})R(x_{0},t_0)$.

\end{prop}

\begin{proof}[Proof of Proposition \ref{p6.3}]

By Corollary 3.6 of \cite{CTZh}, $X$ is diffeomorphic to some $C_{\Gamma}$.
We denote the orbifold universal cover of $X$ by $\tilde{X}$, which is diffeomorphic to $\mathbb{R}^{4}$.
Let $\pi:\tilde{X}\rightarrow{X}$ be the quotient map, and $\tilde{g}_{t}=\pi^{\ast}g_{t}$.
We denote $\tilde{d}_{t_0}$ the distance function with respect to $\tilde{g}_{t_0}$.

%\begin{lem}\label{diam}

%There exists a universal positive constant $\bar{D}$ such that the diameter of $\pi^{-1}(x_{0})$ is bounded from above by $R(x_{0},t_{0})^{-\frac{1}{2}}\bar{D}$.

%\end{lem}

%\begin{proof}[Proof of Lemma \ref{diam}]

From the argument in the proof of Theorem 3.7 of \cite{CTZh}, for $\epsilon>0$ small enough, we can always find a $\Gamma$-invariant $\epsilon$-cap $\tilde{\Omega}$ in $(\tilde{X},\tilde{g}_{t_0})$ such that: \\
\text{}\ \ \ i) every point outside $\tilde{\Omega}$ is a center of a $\Gamma$-invariant $\epsilon$-neck, \\ \text{}\ \ \ ii) there exits a point $\tilde{x}_{0}$ such that for every $\tilde{z}\in\tilde{\Omega}$, $\tilde{d}_{t_0}(\tilde{x}_{0},\tilde{z})$ is bounded from above by $\tilde{R}(\tilde{x}_{0},t_0)^{-\frac{1}{2}}\tilde{D}$, where $\tilde{D}=\tilde{D}(\epsilon)$ is a constant depending on $\epsilon$.

%Therefore, by Proposition 3.6 of \cite{CZh1}, there exists a positive constant $c=c(\epsilon_{6})$ such that $\frac{1}{c}<\tilde{R}(\tilde{y},t_0)\tilde{R}(\tilde{O},t_0)^{-1}<c$ for every $\tilde{y}\in\tilde{\Omega}$.
%Hence if $\pi^{-1}(x_{0})\subset\tilde{\Omega}$, then the diameter of $\pi^{-1}(x_{0})$ is bounded from above by $\tilde{R}(x_{0},t_0)^{-\frac{1}{2}}\tilde{D}_{1}$ for some constant $\tilde{D}_{1}=\tilde{D}_{1}(\epsilon_{6})$.
%If $\pi^{-1}(x_{0})\subset(\tilde{X}\setminus\tilde{\Omega})$, then $\pi^{-1}(x_{0})$ is on the same slice of some $\epsilon_{6}$-neck, and the diameter of $\pi^{-1}(x_{0})$ is bounded from above by $\tilde{R}(x_{0},t_0)^{-\frac{1}{2}}\tilde{D}_{2}$ for some constant $\tilde{D}_{2}$.
%\end{proof}

%Let $\mathcal{C}$ be the $\epsilon$-cap around $(x_0,t_0)$. We choose some $\tilde{x}_{0}\in\pi^{-1}(x_{0})$.
In the following, we fix a sufficiently small $\epsilon$, hence also fix $\tilde{D}=\tilde{D}(\epsilon)$. Rescale the metric $\tilde{g}_{t}$ and shift the time   $t_0$ to $0$ such that ${\tilde{R}}({\tilde{x}}_{0},0)=1$.
Then $\tilde{d}_{0}(\tilde{x}_{0},\tilde{z})\leq\tilde{D}$ for every $\tilde{z}\in\tilde{\Omega}$.
%On the other hand, every point outside the geodesic ball $B_{0}(\tilde{x}_{0},C+\bar{D})$ is a center of an $\epsilon$-neck by assumption.
For every $D>\tilde{D}$, we denote $\tilde{\mathcal{C}}_{D}:=\{\tilde{z}|\tilde{d}_{0}(\tilde{x}_{0},\tilde{z})<{D}\}$.
In the following, we will prove that the smallest eigenvalue of the curvature operator $\tilde{\mathcal{R}}_{0}$ on $\tilde{\mathcal{C}}_{D}$ has a positive lower bound $\tilde{\lambda}(D)$.

Suppose this is not true, then we can find a sequence of ancient $\kappa$-solutions $(M^{(i)},\tilde{g}_{t}^{(i)})$, $t\in(-\infty,0]$, and two sequences of points $\tilde{x}_{0}^{(i)},\tilde{x}^{(i)}\in{M}^{(i)}$ such that:\\
\text{}\ \ \ i) $M^{(i)}$ is diffeomorphic to $\mathbb{R}^{4}$, $\tilde{g}_{t}^{(i)}$ has positive curvature operator,\\
\text{}\ \ \ ii) $\tilde{R}^{(i)}(\tilde{x}_{0}^{(i)},0)=1$,\\
\text{}\ \ \ iii) $\tilde{x}_{0}^{(i)}$ lies in a $\Gamma$-invariant $\epsilon$-cap $\tilde{\Omega}^{(i)}$ with $\tilde{d}_{0}(\tilde{x}_{0}^{(i)},\tilde{z})\leq\tilde{D}$ for every $\tilde{z}\in\tilde{\Omega}^{(i)}$, \\
\text{}\ \ \ iv) $\tilde{d}_{0}(x^{(i)},\tilde{x}_{0}^{(i)})\leq{D}$ and the smallest eigenvalue of the curvature operator at $(\tilde{x}^{(i)},0)$ converges to zero as $i\rightarrow\infty$.

By Theorem 3.5 of \cite{CZh1}, $\tilde{g}_{t}^{(i)}$ are all $\kappa_{0}$-noncollapsed for all scales for some universal $\kappa_{0}$.
By the $\kappa_{0}$-noncollapsed, Proposition 3.6 of \cite{CZh1}, the Li-Yau-Hamilton inequality and Hamilton's compactness theorem, we can extract a $C^{\infty}_{\textmd{loc}}$ converging pointed subsequence, still denote it by $(M^{(i)},\tilde{x}_{0}^{(i)},\tilde{g}_{t}^{(i)})$,
such that the limit $(\check{M}^{4},\check{x}_{0},\check{g}(t))$, $t\in(-\infty,0]$, is a noncompact ancient $\kappa_{0}$-solution to the Ricci flow with nonnegative curvature operator and satisfies the restricted isotropic curvature pinching condition.
On the other hand, $(\tilde{x}^{(i)},0)$ converge to a point $(\check{x},0)$ where the smallest eigenvalue of the curvature operator vanish.
Hence, by Hamilton's strong maximum principle and the restricted isotropic curvature pinching condition, one can prove that $(\check{M}^{4},\check{g}(t))$ is an evolving cylinder $S^{3}\times\mathbb\mathbb\mathbb{R}$ or a noncompact metric quotient of the cylinder (see Lemma 3.2 of \cite{CZh1}).

Recall that a noncompact metric quotient of the round cylinder is isometric to the product metric on $S^{3}/\Gamma\times\mathbb{R}$ or the $\mathbb{Z}_{2}$ quotient of $S^{3}/\Gamma\times\mathbb{R}$ ($\Gamma$ may be trivial).
If $(\check{M}^{4},\check{g}(0))$ is diffeomorphic to $S^{3}/\Gamma\times\mathbb{R}$, then the geodesic ball $B_{0}(\check{x}_{0},R_{0})$ has two ends for large $R_{0}$.
If $(\check{M}^{4},\check{g}(0))$ is diffeomorphic to the $\mathbb{Z}_{2}$ quotient of $S^{3}/\Gamma\times\mathbb{R}$, then the geodesic ball $B_{0}(\check{x}_{0},R_{0})$, for large $R_{0}$, is diffeomorphic to a cap of the form $C^{\tau}_{\Gamma}$, which is not orientable.
On the other hand, if we choose $R_{0}\gg{D>\tilde{D}}$, then for every $i$, $B_{0}(\tilde{x}_{0}^{(i)},R_{0})$ is a cap diffeomorphic to $\R^{4}$ and has one neck-like end. This is a contradiction. This contradiction shows that the estimate is true for the orbifold universal cover. The result follows by passing down to the quotient.
\end{proof}

The following corollary is easy to see:

\begin{cor} \label{posi-cap}

For $\epsilon>0$ small enough and $C>0$, there exists a positive number $\tilde{\epsilon}=\tilde{\epsilon}(\epsilon,C)$, depending on $\epsilon$ and $C$, satisfying the following property. Let $(\mathcal{C},g)$ be a cap  $\tilde{\epsilon}$-close to an $\epsilon$-cap $(\tilde{\mathcal{C}},g_{t_0})$ in an ancient $\kappa$-orbifold solution with positive curvature operator such that there exists some $x_{0}\in\tilde{\mathcal{C}}$ with $d_{t_{o}}(x_{0},z)\leq{R}(x_{0},t_0)^{-\frac{1}{2}}{C}$ for every $z\in\tilde{\mathcal{C}}$. Then $(\mathcal{C},g)$ has positive curvature operator.

\end{cor}
\subsection{Canonical neighborhood assumption}

Let $\epsilon>0$ be a constant,  $r:[0,\infty)\rightarrow\mathbb{R}^{+}$ a non-increasing positive function.
As in \cite{CTZh}, we say that a Ricci flow with surgery $(M,g(t))$, $t\in[0,b]$, satisfies the \textbf{$\epsilon$-canonical neighborhood assumption with parameter $r$,}
if there exist two constants $C_{1}=C_{1}(\epsilon),C_{2}=C_{2}(\epsilon)$ depending only on $\epsilon$ such that every point $(x,t)\in{M}\times[0,b]$ with $R(x,t)\geq{r(t)}^{-2}$ has an open neighborhood $B$, called an $\epsilon$-canonical neighborhood,
satisfying the properties that $B_{t}(x,r)\subset{B}\subset{B}_{t}(x,2r)$ with $0<r<C_{1}R(x,t)^{-\frac{1}{2}}$, and one of the followings:

\begin{description}
  \item[(a)] $B$ is an $\epsilon$-neck around $(x, t)$,
  \item[(b)] $B$ is an $\epsilon$-cap,
  \item[(c)] at time $t$, $x$ lies  in a  compact connected component  with positive curvature operator.
\end{description}
Moreover, for cases (a) and (b), the scalar curvature in $B$ at time $t$ is between $C^{-1}_{2}R(x,t)$ and $C_{2}R(x,t)$, and satisfies the gradient estimate $|\nabla{R}|<\eta{R}^{\frac{3}{2}}$ and $|\frac{\partial{R}}{\partial{t}}|<\eta{R}^{2}$, where $\eta$ is a universal constant.

The existence of canonical neighborhoods of Ricci flow with surgery on four-dimensional orbifolds with PIC  was established in \cite{CZh1} and \cite{CTZh}.

From the proof of Lemma 4.5 in \cite{CTZh} and Proposition 5.4 in \cite{CZh1}, it can be shown that  the $\epsilon$-caps $(\mathcal{C},g)$ in the $\epsilon$-canonical neighborhood assumption actually fall into one of the following types:

\begin{description}
        \item[A] $(\mathcal{C}, g)$ is diffeomorphic to $\mathbb{R}^{4}$ or $C_{\Gamma}$, and $g$ is $\tilde{\epsilon}$-close to a corresponding piece of an ancient $\kappa$-orbifold solution with positive curvature operator $\mathbb{R}^{4}$ or $C_{\Gamma}$, where $\tilde{\epsilon}=\tilde{\epsilon}(\epsilon,2C_{1}(\epsilon))$ is the constant in Corollary \ref{posi-cap}. Hence, $(\mathcal{C}, g)$ has positive curvature operator.
        \item[B] $(\mathcal{C}, g)$ is diffeomorphic to $C_{\Gamma}^{\tau}$ or $C_{II}$, and there exists a double covering $\varphi_{\mathcal{C}}:S^{3}/{\Gamma}\times(-\frac{3}{\epsilon}-10,\frac{3}{\epsilon}+10)\rightarrow\mathcal{C}$ with $\varphi_{\mathcal{C}}(\tau(x),-t)=\varphi_{\mathcal{C}}(x,t)$, or $\varphi_{\mathcal{C}}:S^{3}\times(-\frac{3}{\epsilon}-10,\frac{3}{\epsilon}+10)\rightarrow\mathcal{C}$ with $\varphi_{\mathcal{C}}(\tau(x),-t)=\varphi_{\mathcal{C}}(x,t)$, such that $h^{-2}\varphi_{\mathcal{C}}^{\ast}g$ is  $\epsilon-$ close to $h_{std}$, where $h=R_{g}(z)^{-\frac{1}{2}}$ for some $z\in\varphi_{\mathcal{C}}(S^{3}/{\Gamma}\times\{-\frac{2}{\epsilon}\})$ or $z\in\varphi_{\mathcal{C}}(S^{3}\times\{-\frac{2}{\epsilon}\})$.
        \item[C] $(\mathcal{C}, g)$ is diffeomorphic to $\mathbb{R}^{4}$ or $C_{\Gamma}$, and $g$ is $\epsilon$-close to the corresponding piece of $(\mathbb{R}^{4},g_{std})$ or $(C_{\Gamma},g_{std})$ in the sense that there exists a diffeomorphism $\varphi_{\mathcal{C}}:s_{std}^{-1}((-\frac{3}{\epsilon}, A_{0}])\rightarrow\mathcal{C}$ such that the metric $h^{-2}\varphi_{\mathcal{C}}^{\ast}g$ is $\epsilon-$close to  $g_{std}$, where $h=R_{g}(z)^{-\frac{1}{2}}$ for some $z\in\varphi_{\mathcal{C}}[s_{std}^{-1}(-\frac{2}{\epsilon})]$, $A_{0}$ is the constant such that $(\mathbb{R}^{4}\setminus{B}(p,A_{0}),g_{std})$ is isometric to a half cylinder.
\end{description}

The existence of $\epsilon$-canonical neighborhood can be proved by an argument contradiction as in \cite{CZh1}, \cite{CTZh}. We roughly describe the proof here. The purpose is to indicate to the reader how these types of neighborhoods arise.
If the $\epsilon$-canonical neighborhood assumption is not true, one can find a sequence of Ricci flows with surgery on orbifolds with points $(x_{i},t_{i})$ violating the $\epsilon$-canonical neighborhood assumption with parameters $r_{i}\rightarrow0$, then we rescale the flows around $(x_{i},t_{i})$ so that the scalar curvature at $x_{i}$ becomes $1$.
A delicate argument (see the proof of Lemma 4.5 in \cite{CTZh} and the proof of Proposition 5.4 in \cite{CZh1} for details) guarantees that we can find a subsequence converging to a smooth limit on the maximal time interval $(-t_{\textmd{max}}, 0]$.
If $t_{\textmd{max}}<\infty$, there will be a surgery region near $x_{i}$  for $i$ large, and a neighborhood of $x_{i}$ is close to a piece of $(\mathbb{R}^{4},g_{std})$ or $(C_{\Gamma},g_{std})$, hence is covered by $\epsilon$-necks or $\epsilon$-caps of type $\mathbf{C}$.
If $t_{\textmd{max}}=\infty$, then one can prove that the limiting Ricci flow is an ancient $\kappa$-orbifold solution, and a neighborhood of $x_{i}$ is close to a region in an ancient $\kappa$-orbifold solution for $i$ large.
As described in Theorem \ref{t6.2}, all possible ancient $\kappa$-orbifold solutions can be classified.
If the limiting ancient solution is in case $(1)$, then $x_{i}$ is covered by an $\epsilon$-necks or an $\epsilon$-caps of type $\mathbf{B}$ for $i$ large. If the limit is in case $(2)$, then $x_{i}$ is covered by a compact connected component of positive curvature operator, an  $\epsilon$-necks or an $\epsilon$-caps of type $\mathbf{A}$ for $i$ large.

\section{Proof of Theorem \ref{t1.2}} \label{s7}
Our proof will follow the same strategy  as in \cite{Mar}. Similar results on three dimensional manifolds with positive scalar curvature were obtained by Marques in \cite{Mar}.

We first consider metrics on compact orbifolds where  every point is covered by an $\epsilon$-canonical neighborhood.

\begin{prop} \label{p7.1}

Suppose $(X,g)$ is a compact connected four-orbifold with isolated singularities and with PIC  such that every point $x\in(X,g)$ has an $\epsilon$-canonical neighborhood.  Then $g$ is isotopic to a canonical metric,
and $X$ is diffeomorphic to one of the following orbifolds:

\begin{description}
  \item[1)] $S^{4}/\Gamma$ or $S^{4}/\langle\Gamma,\hat{\tau}\rangle$, where $\Gamma\subset{SO(4)}$ ($\Gamma$ may be trivial),
  \item[2)] $\#(S^{4}/\{1,\zeta\},S^{4}/\{1,\zeta\})$ or $\#(S^{4}/\{1,\zeta\},\mathbb{RP}^{4})$, where the connected sum is performed at smooth points,
  \item[3)] $(S^{3}\times\mathbb{R})/G$, where $G$ a cocompact fixed point free discrete isometric  subgroup of  $(S^{3}\times\mathbb{R},h_{std})$.
\end{description}

\end{prop}
Before the proof of Proposition \ref{p7.1}, we need a lemma:

\begin{lem}\label{interpolateof2neck}
There exists a universal constant $\epsilon_{1}>0$ with the following property. If $0<\epsilon<\epsilon_{1}$, and $N_{1},N_{2}$ are two $\epsilon$-necks in $(X,g)$ with scales $h_{N_{1}},h_{N_{2}}$ respectively.
If there is a point $z\in$ ${s}^{-1}_{N_{1}}(-\frac{0.95}{\epsilon},\frac{0.95}{\epsilon})\bigcap{s}^{-1}_{N_{2}}(-\frac{0.95}{\epsilon},\frac{0.95}{\epsilon})$ with $g(\frac{\partial}{\partial{s_{N_{1}}}},\frac{\partial}{\partial{s_{N_{2}}}})>0$ at $z$.
In addition, ${s}^{-1}_{N_{1}}(-\frac{1}{\epsilon},s_{N_{1}}(z)-2)\bigcap{s}^{-1}_{N_{2}}(s_{N_{2}}(z)+2,\frac{1}{\epsilon})=\emptyset$.
Then there exists a diffeomorphism $\psi:S^{3}/\Gamma\times(-\frac{1}{\epsilon},\beta)\rightarrow{N}_{1}\bigcup{N}_{2}$, $\beta=\frac{1}{\epsilon}-s_{N_{2}}(z)+s_{N_{1}}(z)$, with the following properties:
\begin{description}
  \item[1] $\psi(\theta,t)=\psi_{1}(\theta,t)$ for $(\theta,t)\in{S}^{3}/\Gamma\times(-\frac{1}{\epsilon},s_{N_{1}}(z)-\frac{0.02}{\epsilon})$,
  \item[2] $\psi(\theta,t)=\psi_{2}(A(\theta),t-\beta+\frac{1}{\epsilon})$ for $(\theta,t)\in{S}^{3}/\Gamma\times(\beta+s_{N_{2}}(z)-\frac{0.98}{\epsilon},\beta)$, where $A$ is an isometry of $(S^{3}/\Gamma,d\theta^{2})$,
  \item[3] there exists a continuous path of metrics $g_{\mu},\mu\in[0,1]$, with positive isotropic curvature on $S^{3}/\Gamma\times(-\frac{1}{\epsilon},\beta)$, such that $g_{0}=\psi^{\ast}(g)$, $g_{1}$ is a warped product metric $dt^{2}+\omega(t)^{2}d\theta^{2}$, and such that $g_{\mu}$ restricts to the linear homotopy $g_{\mu}=(1-\mu)\psi^{\ast}(g)+\mu{h}_{N_{1}}^{2}h_{std}$ for $(\theta,t)\in{S}^{3}/\Gamma\times(-\frac{1}{\epsilon},s_{N_{1}}(z)-\frac{0.02}{\epsilon})$ and to the linear homotopy $g_{\mu}=(1-\mu)\psi^{\ast}(g)+\mu{h}_{N_{2}}^{2}h_{std}$ for $(\theta,t)\in{S}^{3}/\Gamma\times(\beta+s_{N_{2}}(z)-\frac{0.98}{\epsilon},\beta)$.
\end{description}

\end{lem}

An $\epsilon$-tube $T\subset{X}$ is a submanifold diffeomorphic to $S^{3}/\Gamma\times(0,1)$ such that every point of the tube is a center of an $\epsilon$-neck.
By suitably decomposing the $\epsilon$-tube into a sequence of overlapped $\epsilon$-necks and then applying Lemma \ref{interpolateof2neck}, we can deform the metric on the $\epsilon$-tube to a warped product metric with a deformation which restricts to linear homotopies at the two ends of the tube.

This lemma is a four-dimensional version of the Lemma 4.1 in \cite{Mar}, the readers can consult Chapter 4 of \cite{Mar} for details of the proof.  The idea is the following.
Suppose $N_{1}$ and $N_{2}$ are two adjacent $\epsilon$-necks satisfying the conditions in Lemma \ref{interpolateof2neck}.
If $\epsilon$ is small enough, the scales of $N_{1}$ and $N_{2}$ are very close to each other, and the two metrics $h_{N_{1}}^{-2}\psi^{\ast}_{N_{1}}(g)$ and $h_{N_{2}}^{-2}\psi^{\ast}_{N_{2}}(g)$ are both small pertubations of $h_{std}$, and the slices of the two necks in the overlap region $\psi_{N_{1}}(S^{3}/\Gamma\times(s_{N_{1}}(z)-4,s_{N_{1}}(z)+4))$ are almost isometric.
Then we can define the diffeomorphism $\psi$ and the deformation of metrics $g_{\mu}$ as in the proof of Lemma 4.1 of \cite{Mar}.

\begin{proof} [Proof of Proposition \ref{p7.1}]
The proof is divided into three cases.

\textbf{Case 1.}   $(X,g)$ is a compact orbifold with positive curvature operator.  Since the normalized Ricci flow converges to a metric of constant curvature (see \cite{Ham3}), we know the space is diffeomorphic to $S^{4}/G.$
Since $X$ has only isolated singularities, by an algebraic lemma of \cite{CTZh}, $X$ is either $S^{4}/\Gamma$ or $S^{4}/\langle\Gamma,\hat{\tau}\rangle$ with $\Gamma\subset SO(4)$ ($\Gamma$ may be trivial).
By Proposition \ref{p4.4}, the constant curvature metric is isotopic to a canonical metric. Hence $g$ is isotopic to a canonical metric.

In the following discussion, we assume that every point of $X$ is covered by an $\epsilon$-neck or an $\epsilon$-cap.

\textbf{Case 2.} Assume that there is no $\epsilon$-cap in $X$. Then every point is contained in an $\epsilon$-neck. Hence $X$ is diffeomorphic to some $S^{3}/\Gamma\times_{f}S^{1}$. Let $N$ be one $\epsilon$-neck with central slice $S$.
Perform surgery around  $S$ and glue surgery caps $C_{\Gamma}$ to both left and right sides of it, we get an orbifold $(\mathcal{S},\tilde{g})$ which is diffeomorphic to $S^4/\Gamma,$ with $\Gamma\subset{SO(4)}$.
If we perform \textmd{M-W} connected sum at the tips of the two surgery caps, by Lemma \ref{l5.5}, the resulting metric is isotopic to $g$.

We claim  $(\mathcal{S},\tilde{g})$ is isotopic to a metric of constant curvature. The reason is that, from Lemma \ref{interpolateof2neck} the $\epsilon$-tube in $\mathcal{S}$ can be deformed to a warped product with PIC with a deformation which restricts to linear homotopies at the two ends of this tube. By Lemma \ref{l5.4}, this deformation extends to the two surgery caps $C_{\Gamma}$.
Hence the metric $\tilde{g}$ is isotopic to a locally conformally flat metric with positive scalar curvature.
The latter is isotopic to a metric of constant curvature by Proposition \ref{p4.5}. Together with Propositions  \ref{orbconnsum}, \ref{p4.10}, we know that $(X,g)$ is isotopic to a canonical metric on $S^{3}/\Gamma\times_{f}S^{1}$.

\textbf{Case 3.} Suppose there exists an $\epsilon$-cap $\mathcal{C}.$ Denote the neck and the core of $\mathcal{C}$ by $N$ and  $Y$.  The neck $N$ is oriented so that the positive $s_{N}$-direction points towards the boundary of the cap. There exist a finite sequence of $\epsilon$-necks $\{N_{1}=N,\ldots,N_{a}\}$ with centers $x_{i}\in{s}_{N_{i}}^{-1}(0),1\leq{i}\leq{a}$, and an $\epsilon$-cap $\tilde{\mathcal{C}}$ with a core $\tilde{Y}$ and an $\epsilon$-neck $\tilde{N}=\tilde{\mathcal{C}}\setminus{\overline{\tilde{Y}}}$, such that\\
\text{}\ \ \ i) $s_{N_{i}}(x_{i+1})=\frac{0.9}{\epsilon}$ and $g(\frac{\partial}{\partial{s_{N_{i}}}},\frac{\partial}{\partial{s_{N_{i+1}}}})>0$ for all $1\leq{i}<a$,\\
\text{}\ \ \ ii) there exists $z\in{s}^{-1}_{N_{a}}(\frac{0.9}{\epsilon})$ contained in the core $\tilde{Y}$,\\
\text{}\ \ \ iii) $X=\mathcal{C}\bigcup{N_{2}}\bigcup\ldots\bigcup{N_{a}}\bigcup\tilde{\mathcal{C}}$.

We will orient the neck $\tilde{N}$ so that the positive $s_{\tilde{N}}$-direction points towards the core $\tilde{Y}$.
It is easy to see $s^{-1}_{\tilde{N}}(0)\subset\overline{Y}\bigcup{s_{N}^{-1}} (-\frac{1}{\epsilon},\frac{0.9}{\epsilon})\bigcup{N_{2}}\bigcup\ldots\bigcup{N_{a-1}}\bigcup {s_{N_{a}}^{-1}}(-\frac{1}{\epsilon},\frac{0.9}{\epsilon})$.
In the following, we denote $\mathcal{C}^{0.9}=\overline{Y}\bigcup{s}_{N}^{-1}((-\frac{1}{\epsilon},\frac{0.9}{\epsilon}))$ and $\tilde{\mathcal{C}}^{0}=\overline{\tilde{Y}}\bigcup{s}_{\tilde{N}}^{-1}((0,\frac{1}{\epsilon}))$.

Recall that we have required $\epsilon$-caps in the canonical neighborhood assumption fall into types $\mathbf{A}$, $\mathbf{B}$ or $\mathbf{C}$.  The argument is divided into the following cases, according to which types of the caps are.

\textbf{Case A-A.} $\mathcal{C}$ and $\tilde{\mathcal{C}}$ are both of type $\mathbf{A}$. In this case, $X$ is diffeomorphic to $S^{4}/\Gamma$ with $\Gamma\subset{SO(4)}.$

If $\partial\tilde{\mathcal{C}}^{0}\subset\mathcal{C}^{0.9}$, then it's easy to see that $\partial\mathcal{C}^{0.9}\subset\tilde{\mathcal{C}}^{0}$,
and $X=\mathcal{C}\bigcup\tilde{\mathcal{C}}$ has positive curvature operator. This reduces to \textbf{Case 1.}

If $\partial\tilde{\mathcal{C}}^{0}$ is not contained in $\mathcal{C}^{0.9}$, then $\partial\tilde{\mathcal{C}}^{0}\subset{s_{N_{k}}^{-1}}((-\frac{0.01}{\epsilon},\frac{0.95}{\epsilon}))$ for some $2\leq{k}\leq{a}$.
We  perform surgeries along the two slices $S_{1}=s_{N_{1}}^{-1}(0)$ and $S_{k}=s_{N_{k}}^{-1}(0.5)$, and glue surgery caps to both left and right sides of each slice as in Section \ref{s5}.
Now $X$ is decomposed into three components all diffeomorphic to $S^{4}/\Gamma.$ Denote them by  $(\mathcal{S}_{1},g_{1})$, $(\mathcal{P},g_{\mathcal{P}})$, and $(\mathcal{S}_{2},g_{2})$.

For $(\mathcal{S}_{1},g_{1})$ and $(\mathcal{S}_{2},g_{2})$, by our assumption of type $\mathbf{A}$ and Proposition \ref{p5.2}, both $g_{1}$ and $g_{2}$ have positive curvature operators, hence are isotopic to canonical metrics on $S^{4}/\Gamma$ as in \textbf{Case 1}.

Now we treat the middle piece $(\mathcal{P},g_{\mathcal{P}}).$ By Lemma \ref{interpolateof2neck}, we can deform the metric $g$ on the $\epsilon$-tube $\bigcup_{i=1}^{k}N_{i}$ to a warped product metric $dr^{2}+\omega(r)^{2}d\theta^{2}$ through PIC metrics
and the deformation restricts to linear homotopies $g_{\mu}=(1-\mu)g+\mu{h}_{N_{1}}^{2}h_{std}$ and $g_{\mu}=(1-\mu)g+\mu{h}_{N_{k}}^{2}h_{std}$ ($\mu\in[0,1]$) on the two ends $s_{N_{1}}^{-1}(-\frac{1}{\epsilon},\frac{0.02}{\epsilon})$ and $s_{N_{k}}^{-1}(\frac{0.02}{\epsilon},\frac{1}{\epsilon})$ respectively.
By Lemma \ref{l5.4},  $(\mathcal{P},g_{\mathcal{P}})$ is isotopic to a locally conformally flat metric with PIC, hence isotopic to a canonical metric on $S^{4}/\Gamma$.

Now we perform \textmd{M-W} connected sum at the tips of the surgery caps of $(\mathcal{S}_{1},g_{1})$, $(\mathcal{P},g_{\mathcal{P}})$, and $(\mathcal{S}_{2},g_{2})$ and obtain a metric $\tilde{g}$ on $X.$ On one hand, by Lemma \ref{l5.5}, $\tilde{g}$ is isotopic to $g$. On the other hand, we have shown that the metrics on these three pieces are isotopic to canonical metrics,  $\tilde{g}$ is isotopic to a canonical metric on $X=S^{4}/\Gamma$ by Proposition \ref{p4.13}.  Therefore, $g$ is isotopic a canonical metric on $S^{4}/\Gamma$.

\textbf{Case C-C.} Both $\mathcal{C}$ and  $\tilde{\mathcal{C}}$ are  of type  $\mathbf{C}$. In this case, $X$ is diffeomorphic to $S^{4}/\Gamma$.  The method is the same as in treating $(\mathcal{S},\tilde{g})$ in the proof of  \textbf{Case 2}.

\textbf{Case A-C.} $\mathcal{C}$ is of type $\mathbf{A}$, $\tilde{\mathcal{C}}$ are of type  $\mathbf{C}$. In this case, $X$ is diffeomorphic to $S^{4}/\Gamma,$  this case can be handled by a cut-past argument and results in \textbf{Cases A-A} and \textbf{C-C}.

\textbf{Case AC-B.} $\mathcal{C}$ is of type $\mathbf{A}$ or $\mathbf{C},$ $\tilde{\mathcal{C}}$ is of type $\mathbf{B}$. In this case, $X$ is diffeomorphic to spherical orbifold $S^{4}/\{1,\zeta\}$ or $S^{4}/\langle\Gamma,\hat{\tau}\rangle$ with $\Gamma\subset{SO(4)}.$
By the assumption on type $\mathbf{B}$, $\mathcal{\tilde{C}}$ is diffeomorphic to $C_{\Gamma}^{\tau}$ or $C_{II}$, and the metric $g$ outside $\mathcal{\tilde{C}}$ can be deformed to a locally conformally flat metric though metrics with positive isotropic curvature, and the deformation restricts to a linear homotopy at the end of $\mathcal{\tilde{C}}$. By a cut-past argument as before, one can show $g$ is isotopic a locally conformally flat metric on the spherical orbifold.

\textbf{Case B-B.} $\mathcal{C}$ and $\tilde{\mathcal{C}}$ are of type $\mathbf{B}$.
In this case, $X$ is diffeomorphic to $\#(S^{4}/\{1,\zeta\},\mathbb{RP}^{4})$, $\#(S^{4}/\{1,\zeta\},S^{4}/\{1,\zeta\})$ or $\#(S^{4}/\langle\Gamma,\hat{\tau}_1\rangle, S^{4}/\langle\Gamma,\hat{\tau}_2\rangle)$. By performing a cut-past argument, the metric $g$ is isotopic to an \textmd{M-W} connected sum of two orbifolds in \textbf{Case AC-B.} By invoking Proposition \ref{p4.13}, $g$ is isotopic to a canonical metric.

The proof is complete.
\end{proof}

Now we begin to prove Theorem \ref{t1.2}.

\begin{proof} [Proof of Theorem \ref{t1.2}]\indent
Let $g_{0}$ be a PIC metric on $M$.
By the theorems in Section 4 of \cite{CTZh}, there exist two sequences of non-increasing small positive numbers  $\{r_{i}\},\{\delta_{i}\}$, and a Ricci flow with surgery on orbifolds with at most isolated singularities $(X_{i},g_{i}(t))_{t\in[t_{i},t_{i+1})}$, $0\leq{i}\leq{p}$, such that:

\begin{description}
  \item[1)] $X_{0}=M$, and $g_{0}(0)=g_{0}$,
  \item[2)] the flow becomes extinct at a finite time $T=t_{p+1}$,
  \item[3)] for every $0\leq{i}\leq{p}$, the flow $(X_{i},g_{i}(t))_{t\in[t_{i},t_{i+1})}$ satisfies the $\epsilon$-canonical neighborhood assumption with parameter $r_{i}$ and pinching assumption,
  \item[4)] for every $0\leq{i}\leq{p-1}$, $(X_{i+1}, g_{i+1}(t_{i+1}))$ is obtained from $(X_{i},g_{i}(t))_{t\in[t_{i},t_{i+1})}$ by doing surgery at  singular time $t_{i+1}$ with parameters $r_{i},\delta_{i}$.
\end{description}

Let $A_{i}$ be the assertion that the restriction of $g_{i}(t_{i})$ to each component of $X_{_{i}}$ is isotopic to a canonical metric. We will prove the theorem by backward induction on $i$.

First, since the flow becomes extinct at $T$, at a time $t'\in[t_{p},T)$ sufficiently close to $T$, every point is covered by  a canonical neighborhood.  By Proposition \ref{p7.1}, $(X_{p},g_{p}(t'))$ is isotopic to a canonical metric.  By Ricci flow equation,  $g_{p}(t_{p})$ is isotopic to $g_{p}(t').$ Hence $g_{p}(t_{p})$ is isotopic to a canonical metric on $X_i,$  and  $A_{p}$ is proven.

In the following, providing  $A_{i+1}$ is true for some $0\leq{i}\leq{p-1}$,  we will prove that $A_{i}$ is true.  Let's recall how the Ricci flow can be extended across the time $t_{i+1}.$ Denote  $g_{i}(t^{-}_{i+1})=\lim_{t\nearrow t_{i+1}}g_{i}(t),$ then $g_{i}(t^{-}_{i+1})$ is a metric with unbounded curvature on $X_i$.

Note that $X_i$ may contain several compact connected components.  For those components of $X_{i}$ with positive curvature operator at time $t_{i+1},$ we know the metrics $g_{i}(t^{-}_{i+1})$ on these components  are isotopic to spherical metrics. Denote the union of the remaining components of $X_i$ by $\Omega$.
 Let $\rho=\delta_{i}r_{i}$ and $\Omega_{\rho}=\{x\in\Omega|\underset{t\rightarrow{t}_{i+1}}{\lim}R(x,t)\leq{\rho}^{-2}\},$
then every point of $\Omega$ outside $\Omega_{\rho}$ has an $\epsilon$-neck or $\epsilon$-cap neighborhood.
There are finite number of connected components of $\Omega\setminus \Omega_{\rho}$ whose one end is in $\Omega_{\rho}$, another end has unbounded curvature. These components are called $\epsilon$-horns and  we denote them  by $H_{j},1\leq{j}\leq{k}.$ Each of these  components is diffeomorphic to $S^{3}/\Gamma\times(0,1)$ for  $\Gamma\subset{SO(4)}.$
By Proposition 4.2 in \cite{CTZh}, there exists $0<h<\delta_{i}\rho_{i}$ such that a point $x_{j}$ on the $\epsilon$-horn $H_{j}$ with curvature $\geq{h}^{-2}$ is a center of a $\delta_{i}$-neck $N_{j}$.
Denote the center slice of $N_{j}$ by $S_{j}$.  Let $\tilde{\Omega}$ be the union of the connected components of $\Omega \setminus\cup S_j$ with finite curvature at $t_{j+1}$ for the metric $g_{i}(t^{-}_{i+1}),$ $\hat{\Omega}=\Omega\setminus\overline{\tilde{\Omega}}$.
Now we cut off the $\delta_{i}$-neck $N_{j}$ along $S_{j},$ and glue back caps to boundary necks of $\tilde{\Omega}.$  We denote the resulting orbifold by $C_{i+1}.$  Clearly,  $C_{i+1}$ is just $(X_{i+1},g_{i+1}(t_{i+1}))$.

On the other hand, for $t'\in(t_{i},t_{i+1})$ sufficiently close to $t_{i+1}$, the family of metrics  $(1-\mu)g_{i}(t')+\mu{g}_{i+1}(t_{i+1})$ ($\mu\in[0,1]$) on the curvature finite part $\tilde{\Omega}\cup\cup_{j} N_j$ have positive isotropic curvature and has $\delta_{i}$-neck structures on each $N_{j}$.
Gluing surgery caps at the slices $S_{j}$ on this family of metrics, we know that $(X_{i+1},g_{i}(t')_{surg})$ is isotopic to $(X_{i+1}, g_{i+1}(t_{i+1}))$. By induction assumption $A_{i+1}$, on $X_{i+1}$, $g_{i}(t')_{surg}$ is isotopic to a canonical metric.

Moreover, at time $t',$ if we glue back caps to the boundary necks of $\hat{\Omega},$ we get a (possibly disconnected) closed orbifold $(Y_{i+1}, g_{i}(t')_{surg})$. Every point of $Y_{i+1}$  has a canonical neighborhood, by Proposition \ref{p7.1}, on each connected component of $Y_{i+1}$, $g_{i}(t')_{surg}$ is isotopic to a canonical metric.

Finally, by Lemma \ref{l5.5}, if we perform  \textmd{M-W} connected sums at the tips of the surgery caps of $(X_{i+1},g_{i}(t')_{surg})$ and $(Y_{i+1}, g_{i}(t')_{surg})$, the resulting metric on $\Omega$ is isotopic to $g_{i}(t')$. Hence the metric $g_{i}(t')$ on $\Omega$ is isotopic to a canonical metric by Proposition \ref{p4.13}.
This proves $A_{i}$.

Repeating the above procedure, we know that $g_{0}$ is isotopic to a canonical metric on $M$. Since $M$ is itself a manifold, we know there is no subcomponent (in the canonical decomposition) containing singular points. Therefore, every subcomponent is either diffeomorphic to $\mathbb{RP}^{4}$ or $(S^{3}\times\mathbb{R})/G$. The proof is complete.
\end{proof}

\section{Proof of Theorem \ref{t1.1}} \label{s8}

Let $(M,g)$ be a compact orientable four-manifold with positive isotropic curvature. Then by \cite{CTZh}, we know $M$ is diffeomorphic to a connected sum $S^4\# M_1\# M_2\cdots \# M_k$, where $M_i$ is diffeomorphic to $S^3/\Gamma_i\times_{f_i} S^1$.
We will first show that this decomposition is unique in the fundamental group level.

\begin{thm} \label{t8.1}
Suppose $(M,g)$ admits two canonical decompositions $S^4\# M_1\# M_2\cdots \# M_k$ and $S^4\# N_1\# N_2\cdots \# N_l$, then $k=l,$ and there is a permutation $\sigma\in S_k$ such that $\pi_{1}(M_i)\cong \pi_{1}(N_{\sigma_i})$ for all $1\leq i \leq k$.
\end{thm}

Recall that a group $G$ is called freely indecomposable if $G$ does not admit a nontrivial free product decomposition $G=A\ast{B}$ with $A,B$ proper subgroups of $G$.
\begin{lem} \label{l8.2}

Suppose $M$ is orientable and diffeomorphic to ${S}^{3}/\Gamma\times_{f}S^{1}$, then $\pi_{1}(M)$ is freely indecomposable.

\end{lem}

\begin{proof}

Let $G=\pi_{1}(M)$ be the deck transformation group acting isometrically on $S^{3}\times\mathbb{R}.$ Then  $\Gamma=G\bigcap\textmd{Isom}(S^{3}\times\{0\})$ is a finite normal subgroup of $G$, and $G$ is isomorphic to a semidirect product $G=\Gamma\rtimes\mathbb{Z}$.
Suppose $G=A\ast{B}$, where $A$ and $B$ are both proper subgroups of $G$.
We select nontrivial elements $a\in{A}$, $b\in{B}$, and suppose they have  the forms  $a:(x,t)\rightarrow(h(x),t+2k\pi)$ and $b:(x,t)\rightarrow(\tilde{h}(x),t+2l\pi)$, where $h,\tilde{h}\in\textmd{Isom}(S^{3})$, $k,l\in\mathbb{Z}$.
The commutator $[a,b]=aba^{-1}b^{-1}:(x,t)\rightarrow(h\tilde{h}h^{-1}\tilde{h}^{-1}(x),t)$, i.e. $aba^{-1}b^{-1}\in\Gamma$.
But $\Gamma$ is a finite group, then there exist a positive number $n$ such that $(aba^{-1}b^{-1})^{n}=id$, this contradicts the assumption that there is no relation between the elements of $A$ and $B$.
Hence $G$ is freely indecomposable.
\end{proof}

The proof of Theorem \ref{t8.1} relies on the following  Kurosh's  Subgroup Theorem (see \cite{Ohsh}):

\begin{thm}\label{t8.3}
Suppose that a group $G$ admits two free product decompositions $G_{1}\ast\ldots\ast{G}_{k}$, $G'_{1}\ast\ldots\ast{G}'_{l}$, each of whose factors is non-trivial and freely indecomposable. Then these two decompositions are isomorphic in the sense that $k=l$ and there is a permutation $\sigma$ of $\{1,\ldots,k\}$ such that $G_{i}\cong{G}'_{\sigma(i)}$ for all $1\leq i \leq k$.

\end{thm}
\begin{proof} [Proof of Theorem \ref{t8.1}] By van Kampen Theorem, $\pi_{1}(M)\cong \pi_{1}(M_1)\ast\cdots \ast \pi_{1}(M_k)\cong \pi_{1}(N_1)\ast\cdots \ast \pi_{1}(N_l).$ The result then follows from Theorem \ref{t8.3} and Lemma \ref{l8.2}.
\end{proof}

Suppose $M$ is diffeomorphic to $S^4\# M_1\# M_2\cdots \# M_k$, where $M_i$ is diffeomorphic to $S^3/\Gamma_i\times_{f_i} S^1$.
When $k=0$, $M$ is diffeomorphic to $S^4$. Since the canonical metric on $S^4$ is the standard round metric, the result is clearly true in this case.
Now we handle the $k=1$ case.

\begin{prop} \label{p8.4} Let $M$ be an orientable four manifold equipped with two canonical metrics $\tilde{g}_i$ and the associated canonical decompositions  $S^4\# M^{i}$ has only one nontrivial piece.   Then there is a diffeomorphism $\varphi$ such that $\tilde{g}_1$ is isotopic to $\varphi^{*}\tilde{g}_2.$
 \end{prop}

Before the proof of  Proposition \ref{p8.4}, we need some preliminary results on fiber bundles.  For two fiber bundles $\pi_{1}:E_{1}\rightarrow{B_{1}}$ and $\pi_{2}:E_{2}\rightarrow{B_{2}}$ with the same fiber $F$,  we call $E_{1}$ and $E_{2}$ are weakly equivalent if there exists a differmorphism $\phi:B_{1}\rightarrow{B}_{2}$ and a bundle map $\Phi:E_{1}\rightarrow{E_{2}}$ such that the following diagram commutes:

\[
\xymatrix{
  E_{1} \ar[d]_{\pi_{1}} \ar[r]^{\Phi} & E_{2} \ar[d]^{\pi_{2}} \\
  B_{1} \ar[r]^{\phi} & B_{2}   }
\]
The above $\Phi$ is called a weak bundle isomorphism. Roughly speaking, a weak bundle isomorphism is a diffeomorphism of the total space which preserves the fibers. In the proof of of Proposition \ref{p8.4} we need the following theorem:

\begin{prop}[Proposition 8 of \cite{Ue}, see also \cite{Kato} and \cite{Ohba}] \label{t8.5}

Let $M_{i}=F_{i}\times_{f_{i}}{S}^{1},$ $i=1,2,$ be two bundles over $S^{1}$  with fibers $F_{i}$ and  monodromy $f_{i}$, where $F_{i}$ are spherical $3$-manifolds. Then there is a weak bundle isomorphism between $M_1$ and $M_2,$ if one of the following two conditions holds: \\
\text{}\ \ \ i) $M_{1}$ is diffeomorphic to $M_{2}$ as differential manifolds,\\
\text{}\ \ \ ii)  $\pi_{1}(M_{1})\cong\pi_{1}(M_{2})$ and  $F_{1}$ is diffeomorphic to $F_{2}$.

%Moreover $M_{1}$ is diffeomorphic to $M_{2}$ if and only if $\pi_{1}(M_{1})=\pi_{1}(M_{2})$ unless $M_{i}=L_{i}\times_{f_{i}}S^{1}$ for some lens space $L_{i}(i=1,2)$.
\end{prop}
\begin{proof} [Proof of Proposition \ref{p8.4}]

A diffeomorphism $\varphi:M\rightarrow{S}^{3}/\Gamma\times_{f}S^{1}$ with $f\in\textmd{Isom}(S^{3}/\Gamma)$ induced a bundle structure with base $S^{1}$, fiber $F={S}^{3}/\Gamma$ and monodromy $f$.
On the other hand, by Proposition \ref{t8.5}, if there are two diffeomorphisms $\varphi_{1}:M\rightarrow{S}^{3}/\Gamma_{1}\times_{f_{1}}S^{1}$ and $\varphi_{2}:M\rightarrow{S}^{3}/\Gamma_{2}\times_{f_{2}}S^{1}$,
then we have a weak bundle isomorphism between ${S}^{3}/\Gamma_{1}\times_{f_{1}}S^{1}$ and ${S}^{3}/\Gamma_{2}\times_{f_{2}}S^{1}$.
By the theory of fiber bundles, one can prove that the fiber bundles $S^{3}/\Gamma\times_{f_{1}}S^{1}$ and $S^{3}/\Gamma\times_{f_{2}}S^{1}$ are weakly equivalent if and only if there exists a $\beta\in\textmd{Diff}(S^{3}/\Gamma)$ such that $f_{1}$ (or $f_{1}^{-1}$)  is isotopic to  $\beta{f}_{2}\beta^{-1}$ in $\textmd{Diff}(S^{3}/\Gamma)$  (see Section 18 of \cite{Steen}).
Since any element in  $\textmd{Diff}(S^{3}/\Gamma)$ is isotopic to an element in $\textmd{Isom}(S^{3}/\Gamma).$ We can assume the above $f_{1},f_{2},\beta\in\textmd{Isom}(S^{3}/\Gamma)$.
For $f\in\textmd{Isom}(S^{3}/\Gamma)$, let $G_{f}$ be the isometric subgroup of $(S^{3}/\Gamma)\times\mathbb{R}$ generated by $\rho:(\theta,s)\rightarrow(f(\theta),s+2\pi).$ Let $S^{3}/\Gamma\times_{f}S^{1}$ is equipped with the quotient metric $g_{f}$ induced from $(S^{3}/\Gamma)\times\mathbb{R}$.
For $f,\beta\in\textmd{Isom}(S^{3}/\Gamma)$, we define a map $\iota:S^{3}/\Gamma\times_{f}S^{1}\rightarrow{S}^{3}/\Gamma\times_{\beta{f}\beta^{-1}}S^{1}$ by $\iota(\theta,s)=(\beta(\theta),s)$,
where $(\theta,s)$ is the coordinate of $S^{3}/\Gamma\times_{f}S^{1}$ and $S^{3}/\Gamma\times_{\beta{f}\beta^{-1}}S^{1}$ induced from $(S^{3}/\Gamma)\times\mathbb{R}$ respectively.
It is easy to verify that $\iota$ is well defined, and $\iota^{\ast}(g_{\beta{f}\beta^{-1}})=g_{f}$.
Similarly, the map $\chi:S^{3}/\Gamma\times_{f}S^{1}\rightarrow{S}^{3}/\Gamma\times_{\beta{f}^{-1}\beta^{-1}}S^{1}$ defined  by $\chi(\theta,s)=(\beta(\theta),-s)$ also satisfies  $\chi^{\ast}(g_{\beta{f^{-1}}\beta^{-1}})=g_{f}$.
\end{proof}

\begin{cor} \label{c8.6} Let $M$ be an orientable manifold diffeomorphic to $S^3/\Gamma\times_{f} S^1$. Then  $\textmd{PIC}(M)/\textmd{Diff}(M)$ is path-connected.
\end{cor}

When $k\geq2$, we have the following theorem:

%\begin{thm}\label{t8.7}
%Let $M$ be a manifold diffeomorphic to $S^3\times S^1\#\cdots \# S^3\times S^1$. Then $\textmd{PIC}(M)/\textmd{Diff}(M)$ is path-connected.
%\end{thm}

%\begin{proof}
%As before, let $g_1,g_2\in\textmd{PIC}(M)$, then by Theorem \ref{t1.2}, there are two canonical metrics $\tilde{g}_1,\tilde{g}_2\in\textmd{PIC}(M)$ such that $g_i$ is isotopic to $\tilde{g}_{i}$ for $i=1,2$.
%Let $S^4\# M_1\# M_2\cdots \# M_k,$  and $S^4\# N_1\# N_2\cdots \# N_k,$ be the associated canonical decompositions respectively. Denote by $(M_j,h_j,\sigma_j)$ the standard metric $h_j$ and orientation $\sigma_j$ on $M_j$ inherited from $M$ in the above decompositions. Denote the corresponding quantity on $N_j$ be $(N_j,h_j',\sigma_j')$.
%From Theorem \ref{t8.3}, we know $\pi_{1}(M_i)=\pi_1(N_j)=\mathbb{Z}$. Then we know both $M_i$ and $N_j$ are diffeomorphic to $S^3\times S^1$.
%Note that there is an orientation reversing isometry of the canonical metric on $S^3\times S^1.$  So there always is an orientation preserving diffeomorphism $\varphi_j$ from $(M_j,\sigma_j)$ to $(N_j,\sigma_j')$ such that $\varphi_j^{*}h_j'$ is isotopic to $h_j.$
%From Proposition \ref{orbconnsum}, during making \textmd{M-W} connected sums, these $\varphi_j$'s may be glued together to give a global diffeomorphism $\varphi$ from $M$ to itself, such that $\varphi^{*}\tilde{g}_2$ is isotopic to $\tilde{g}_1.$ From this, we know $g_1$ is isotopic to $\varphi^{*}{g}_2.$ The proof is complete.
%\end{proof}

\begin{thm}\label{t8.7}
Let $M$ be a manifold diffeomorphic to a finite connected sum of $S^{3}/\Gamma_{i}\times S^{1}$, $1\leq{i}\leq{k}$, where $\Gamma_{i}$ is either the trivial group or a non-cyclic discrete isometric group of $S^{3}$. Then $\textmd{PIC}(M)/\textmd{Diff}(M)$ is path-connected.
\end{thm}

\begin{proof}
As before, let $g_1,g_2\in\textmd{PIC}(M)$, then by Theorem \ref{t1.2}, there are two canonical metrics $\tilde{g}_1,\tilde{g}_2\in\textmd{PIC}(M)$ such that $g_i$ is isotopic to $\tilde{g}_{i}$ for $i=1,2$.
Let $S^4\# M_1\# M_2\cdots \# M_k,$  and $S^4\# N_1\# N_2\cdots \# N_k,$ be the associated canonical decompositions respectively. Denote by $(M_j,h_j,\sigma_j)$ the standard metric $h_j$ and orientation $\sigma_j$ on $M_j$ inherited from $M$ in the above decompositions. Denote the corresponding quantity on $N_j$ by $(N_j,h_j',\sigma_j')$.
From Theorem \ref{t8.3}, we may assume $\pi_{1}(M_j)=\pi_1(N_j)=\Gamma_{j}\times\mathbb{Z}$ for $1\leq{j}\leq{k}$.
Note that if two spherical 3-manifolds $F_{1},F_{2}$ satisfy that $\pi_{1}(F_{1})=\pi_{1}(F_{2})$ are trivial or non-cyclic, then $F_{1},F_{2}$ are diffeomorphic to each other (see \cite{Mcc}).
Combining this fact with ii) in Proposition \ref{t8.5} and Proposition \ref{p8.4}, both $M_j$ and $N_j$ are diffeomorphic to $S^{3}/\Gamma_{j}\times S^{1}$, and the standard metrics on them are isotopic to the standard product metrics (hence we may assume $h_j$ and $h_j'$ are product metrics).
Note that there is an orientation reversing isometry of the standard product metric on $S^3/\Gamma\times S^1$ defined by $\chi:S^3/\Gamma\times S^1\rightarrow S^3/\Gamma\times S^1,\chi(\theta,s)=(\theta,-s)$.
So there always is an orientation preserving diffeomorphism $\varphi_j$ from $(M_j,\sigma_j)$ to $(N_j,\sigma_j')$ such that $\varphi_j^{*}h_j'$ is isotopic to $h_j$.
From Proposition \ref{orbconnsum}, during making \textmd{M-W} connected sums, these $\varphi_j$'s may be glued together to give a global diffeomorphism $\varphi$ from $M$ to itself, such that $\varphi^{*}\tilde{g}_2$ is isotopic to $\tilde{g}_1$. From this, we know $g_1$ is isotopic to $\varphi^{*}{g}_2$. The proof is complete.
\end{proof}

Theorem \ref{t1.1} follows from Theorem \ref{t8.7}, Corollary \ref{c8.6} and Proposition \ref{p4.6}.

\begin{rem}
In dimension three, any compact orientable manifold $M^3$ admits a prime decomposition $M^3=M_1\#\cdots \# M_k,$ moreover, this decomposition is unique up to order and homeomorphism. The existence part is due to Kneser \cite{Kn}, and the uniqueness part is due to Milnor \cite{Mil}.
When dimension equals to four, the uniqueness of such a  decomposition is not true in general, for instance, Hirzebruch proved (see \cite{Mil}) that $(S^2\times S^2)\#\overline{\mathbb{CP}}^2$ is diffeomorphic to $\mathbb{CP}^2\#\overline{\mathbb{CP}}^2\#\overline{\mathbb{CP}}^2$.
This is a big difference between dimension 3 and 4. For the sake of completely solving the isotopy problem of PIC metrics, the technique employed in this paper relies on the uniqueness of the canonical decompositions for manifolds with PIC metrics.
\end{rem}

\end{document}